\documentclass[a4paper,hidelinks,oneside,onefignum,onetabnum]{article}
\usepackage{amsmath,amssymb,amsthm,bm,bbm,mathrsfs,mathtools,graphicx}
\usepackage{algorithm,algorithmic,enumerate}
\usepackage[a4paper,paperheight=297mm,paperwidth=210mm,margin=1.25in]{geometry} % 设置边距为1英寸  
\usepackage{authblk}
\usepackage[font=small,labelfont=md,textfont=it]{caption}
\usepackage[titletoc, title]{appendix}
\usepackage{enumitem}
\usepackage[colorlinks,linkcolor=black,citecolor=black]{hyperref}
\usepackage{etoolbox}
\usepackage{longtable}
\usepackage{diagbox}
\usepackage{booktabs,makecell,multirow}
\usepackage[capitalize,nameinlink]{cleveref}[0.19]
\usepackage{cases,color}
\crefname{lemma}{Lemma}{Lemmas}
\crefname{theorem}{Theorem}{Theorems}
\crefname{discr}{Discretization}{Discretizations}

\crefformat{equation}{\textup{#2(#1)#3}}
\crefrangeformat{equation}{\textup{#3(#1)#4--#5(#2)#6}}
\crefmultiformat{equation}{\textup{#2(#1)#3}}{ and \textup{#2(#1)#3}}
{, \textup{#2(#1)#3}}{, and \textup{#2(#1)#3}}
\crefrangemultiformat{equation}{\textup{#3(#1)#4--#5(#2)#6}}%
{ and \textup{#3(#1)#4--#5(#2)#6}}{, \textup{#3(#1)#4--#5(#2)#6}}{, and \textup{#3(#1)#4--#5(#2)#6}}

\Crefformat{equation}{#2Equation~\textup{(#1)}#3}
\Crefrangeformat{equation}{Equations~\textup{#3(#1)#4--#5(#2)#6}}
\Crefmultiformat{equation}{Equations~\textup{#2(#1)#3}}{ and \textup{#2(#1)#3}}
{, \textup{#2(#1)#3}}{, and \textup{#2(#1)#3}}
\Crefrangemultiformat{equation}{Equations~\textup{#3(#1)#4--#5(#2)#6}}%
{ and \textup{#3(#1)#4--#5(#2)#6}}{, \textup{#3(#1)#4--#5(#2)#6}}{, and \textup{#3(#1)#4--#5(#2)#6}}

\crefdefaultlabelformat{#2\textup{#1}#3}

\apptocmd{\sloppy}{\hbadness 10000\relax}{}{}
\newcommand{\dual}[1]{\langle {#1} \rangle}
\newcommand{\dualb}[1]{\big\langle {#1} \big\rangle}
\newcommand{\dualB}[1]{\Big\langle {#1} \Big\rangle}

\newcommand{\norm}[1]{\lVert {#1} \rVert}
\newcommand{\normb}[1]{\big\lVert {#1} \big\rVert}
\newcommand{\normB}[1]{\Big\lVert {#1} \Big\rVert}

\newcommand{\snm}[1]{\lvert {#1} \rvert}

\newcommand{\snmB}[1]{\Big\lvert {#1} \Big\rvert}

\newcommand{\ssnm}[1]
{
	\left\vert\kern-0.25ex
	\left\vert\kern-0.25ex
	\left\vert
	{#1}
	\right\vert\kern-0.25ex
	\right\vert\kern-0.25ex
	\right\vert
}

\makeatletter
\def\spher@harm#1{%
	\vbox{\hbox{%
			\offinterlineskip
			\valign{&\hb@xt@2\p@{\hss$##$\hss}\vskip.2ex\cr#1\crcr}%
		}\vskip-.36ex}%
}
\def\gshone{\spher@harm{.}}
\def\gshtwo{\spher@harm{.&.}}
\def\gshthree{\spher@harm{.&.&.}}
\let\gsh\spher@harm
\makeatother

\newtheorem{proposition}{Proposition}[section]

\newtheorem{lemma}{Lemma}[section]
\newtheorem{remark}{Remark}[section]
\newtheorem{theorem}{Theorem}[section]

\numberwithin{equation}{section}

\makeatletter\def\@captype{table}\makeatother

\begin{document}
% \title{
%   \Large\bf Pathwise uniform convergence of a full discretization for a
%   three-dimensional stochastic Allen-Cahn equation with
%   multiplicative noise
%   \thanks{
%     This work was partially supported by National Natural
%     Science Foundation of China under grant 12301525 and
%     Natural Science Foundation of Sichuan Province under grant 2023NSFSC1324.
%     % This work was partially supported by
%     % Natural Science Foundation of Sichuan province under grant 2023NSFSC1324.
%   }
% }

\title{
  \Large\bf Pathwise uniform convergence of a full discretization for a
  three-dimensional stochastic Allen-Cahn equation with
  multiplicative noise
}
    % This work was partially supported by
    % Natural Science Foundation of Sichuan province under grant 2023NSFSC1324.

% \author{Binjie Li\thanks{libinjie@scu.edu.cn} }
% \author{Qin Zhou\thanks{Corresponding author: }}

\author[1]{Binjie Li\thanks{libinjie@scu.edu.cn}}
% \author{Binjie Li\thanks{libinjie@scu.edu.cn}}
% \author[2]{Qin Zhou\thanks{Corresponding author: zhouqinmath@stu.scu.edu.cn}}
\author[2]{Qin Zhou\thanks{Corresponding author: zqmath@cwnu.edu.cn}}
\affil[1]{School of Mathematics, Sichuan University, Chengdu 610064, China}
\affil[2]{School of Mathematics, China West Normal University, Nanchong 637002, China}
% \affil{School of Mathematics, Sichuan University, Chengdu 610064, China}

% \affil{School of Mathematics, Sichuan University, Chengdu 610064, China}

% \date{\today}
\date{}
\maketitle

% This paper analyzes a full discretization of a three-dimensional stochastic Allen-Cahn equation
% with multiplicative noise. The discretization uses the Euler scheme for temporal discretization
% and the finite element method for spatial discretization.
% By employing both discrete deterministic and stochastic maximal \( L^p \)-regularity estimates,
% the study derives a pathwise uniform convergence rate that encompasses general spatial \( L^q \)-norms.
% Moreover, the theoretical convergence rate is verified through numerical experiments.

\begin{abstract}
This paper analyzes a full discretization of a three-dimensional stochastic Allen-Cahn equation
with multiplicative noise. The discretization combines the Euler scheme for temporal approximation
and the finite element method for spatial approximation.
A pathwise uniform convergence rate is derived, encompassing general spatial \( L^q \)-norms,
by using discrete versions of deterministic and stochastic maximal \( L^p \)-regularity estimates.
Additionally, the theoretical convergence rate is validated through numerical experiments.
The primary contribution of this work is the introduction of a technique to establish the
pathwise uniform convergence of finite element-based full discretizations for nonlinear
stochastic parabolic equations within the framework of general spatial \( L^q \)-norms.
\end{abstract}


\section{Introduction}
Let \( T \in (0, \infty) \) denote the terminal time, and let \( \mathcal{O} \) be a bounded, convex domain in \( \mathbb{R}^3 \) with a smooth boundary \( \partial\mathcal{O} \). Let \( \Delta \) represent the Laplace operator subject to homogeneous Dirichlet boundary conditions. This paper studies the following stochastic Allen-Cahn equation:
\begin{equation}
\label{eq:model}
\begin{cases}
  \mathrm{d}y(t) = \left( \Delta y + y - y^3 \right)(t) \, \mathrm{d}t + F(y(t)) \, \mathrm{d}W_H(t), & t \in [0, T], \\
  y(0) = v,
\end{cases}
\end{equation}
where \( v \) denotes the prescribed initial condition,a
and \( W_H \) represents an \( H \)-cylindrical Brownian motion.
The diffusion coefficient \( F \) will be precisely defined later.

The Allen-Cahn equation is an important mathematical model in
materials science~\cite{Allen-Cahn1979} and it is also of fundamental
importance for the geometric moving interface problem~\cite{Evans1992}.
By incorporating thermal fluctuations, impurities, and atomistic processes
that describe surface motion, the stochastic Allen-Cahn equation
naturally arises~\cite{Funaki1995,Kohn2007,Roger2013}.

Recently, there has been a growing interest in the numerical analysis
of the stochastic Allen-Cahn equation. A brief summary of some recent
works is given below. For the stochastic Allen-Cahn equation
with additive noise, Becker and Jentzen~\cite{Becker2019} obtained
strong convergence rates for a nonlinearity-truncated Euler-type
approximation. Becker et al.~\cite{Becker2023} established sharp
strong convergence rates for an explicit space-time discrete
numerical approximation. Br\'ehier et al.~\cite{Brehier2019IMA} analyzed an
explicit temporal splitting numerical scheme. Qi and Wang~\cite{QiWang2019} derived
strong error estimates for a full discretization using the finite element
method for spatial discretization and the Euler scheme for temporal discretization.
In addition, Wang~\cite{Wangxiaojie2020} analyzed an efficient explicit
full-discrete scheme for a one-dimensional stochastic Allen–Cahn equation.
The numerical analysis of the stochastic Allen-Cahn equation with multiplicative
noise generally presents greater challenges. In this regard,
Feng et al.~\cite{Feng2017SIAM} derived the strong convergence of finite element
approximations of a stochastic Allen-Cahn equation with finite-dimensional
gradient-type multiplicative noise, Majee and Prohl~\cite{Prohl2018} analyzed
a space-time discretization, and Liu and Qiao~\cite{LiuQiao2021} investigated the strong convergence of Galerkin-based Euler approximations.
We also refer the reader to~\cite{Gyongy2016,Hutzenthaler2018,Jentzen2009,
Jentzen2018,Jentzen2020IMA,Kovacs2018,Kovacs2018-2,kovacs2011,QiZhangXu2023,Sauer2015} and
the references therein for more related works.

So far, the numerical analysis of the stochastic Allen-Cahn equation in the literature has generally considered only the $ L^p(\Omega;L^2) $-norm error estimates for fully discrete approximations. There is a lack of pathwise uniform convergence estimates and convergence estimates with general spatial $ L^q $-norms. Notably, Bréhier et al. \cite{Brehier2019IMA} derived some pathwise uniform convergence estimates for an explicit temporal splitting numerical scheme; however, their analysis might not be directly applicable to general domains (see \cite[pp.~2116]{Brehier2019IMA}). Extending this analysis to the full discretization of the three-dimensional stochastic Allen-Cahn equation with multiplicative noise is nontrivial.

% the numerical analysis of the stochastic Allen-Cahn equation
% in the literature generally considers only
% \[
%   \max_{1 \leqslant j \leqslant J}
%   \norm{y(t_j) - Y_j}_{L^p(\Omega;L^2)}, \, p \in [1,\infty),
% \]
% where $ y $ and $ \{Y_j\}_{j=1}^J $ are the solutions of the continuous
% problem and the discretization, respectively. There is a lack of pathwise
% uniform convergence estimates and convergence estimates with the general
% spatial $L^q$-norms of the fully discrete approximations. It is worth noting
% that Br\'ehier et al.~\cite{Brehier2019IMA} derived some pathwise uniform
% convergence estimates for an explicit temporal splitting numerical scheme;
% however, it seems that their analysis may not be applicable to general
% convex bounded domains (see \cite[pp.~2116]{Brehier2019IMA}),
% and extending this analysis to the full discretization of the three-dimensional
% stochastic Allen-Cahn equation with multiplicative noise may not be straightforward.

The main contributions of this paper are summarized as follows:
\begin{itemize}
  \item We establish the stability estimate
    \[
      \mathbb{E}\left[\max_{1 \leqslant j \leqslant J}
      \left\| \sum_{k=0}^{j-1} \int_{t_k}^{t_{k+1}} (I - \tau \Delta_h)^{k-j} g_h(t) \, \mathrm{d}W_H(t) \right\|_{L^q}^p \right] \leqslant C \left\| g_h \right\|_{L^p(\Omega \times (0,T); \gamma(H, L^q))}^p,
    \]
    for all \( g_h \in L_\mathbb{F}^p(\Omega \times (0,T); \gamma(H, X_h)) \),
    where \( p \in (2, \infty) \) and \( q \in [2, \infty) \).
    Here, \( X_h \) denotes the finite element space and $ L^q $ denotes $ L^q(\mathcal O) $.
    This estimate is useful for demonstrating pathwise uniform convergence of
    fully discrete approximations based on finite elements for nonlinear stochastic parabolic equations,
    within the framework of general spatial \( L^q \)-norms.
    While van Neerven and Veraar \cite{Neerven2022} have developed analogous estimates for
    the Euler scheme, their methodology does not directly apply to deriving the aforementioned
    stability estimate, as detailed in Subsection \ref{subsec:stability}.
  \item For a fully discrete approximation of \cref{eq:model},
    using the Euler scheme for temporal discretization and the finite
    element method for spatial discretization, we establish the following
    pathwise uniform convergence estimate:
    \[
      \normB{
        \max_{1 \leqslant j \leqslant J}
        \norm{y(t_j) - Y_j}_{L^q}
      }_{L^p(\Omega)} \leqslant
      c \left( h^{2-\varepsilon} + \tau^{1/2} \right)
    \]
    for all $ p\in(2,\infty) $, $ q \in [2,\infty) $ and $\varepsilon>0 $.
    Notably, even for the three-dimensional stochastic Allen-Cahn equation with additive noise,
    to the best of our knowledge, such error estimates for fully discrete approximations using
    the Euler scheme have not been reported previously. This finding fills an important gap in
    the existing literature.
\end{itemize}

Our analysis is based on discrete deterministic and stochastic maximal $ L^p $-regularity estimates. We impose relatively mild regularity conditions on the diffusion coefficient $ F $ in \cref{eq:model}. Although our analysis assumes that the initial value $ v $ is deterministic and belongs to $ W_0^{1,\infty}(\mathcal{O}) \cap W^{2,\infty}(\mathcal{O}) $, it can be extended to cases where the initial value $ v $ has lower regularity. 
In addition, the pathwise uniform convergence of numerical approximations for
SPDEs within the framework of Banach spaces has recently attracted considerable
attention; see, for example, \cite{Cox2010, Cox2013, Klioba2024, Neerven2022}.
However, existing research has primarily focused on temporal semidiscretizations. This work introduces a new approach for establishing pathwise uniform convergence for finite element-based full discretizations of nonlinear stochastic parabolic equations, within the general spatial $ L^q $-norms framework.

The remainder of this paper is organized as follows.
In Section \ref{sec:math}, we introduce the notational conventions used throughout this work
and present a regularity result for the model problem.
In Section \ref{sec:main}, we describe a full discretization and provide a pathwise uniform
convergence estimate.
In Section \ref{sec:proof}, we establish a stability estimate for a discrete stochastic
convolution and provide a proof for the previously mentioned convergence estimate.
Section \ref{sec:numer} presents numerical experiments to validate the theoretical result.
Finally, Section \ref{sec:conclusion} summarizes the findings and contributions of this study.

% Our analysis heavily relies on the discrete stochastic maximal $ L^p $-regularity
% developed in \cite{LiXieLpTime2023,LiLpSpatail2023}, which was inspired by the 
% theory of stochastic maximal $ L^p $-regularity in~\cite{Neerven2012}. A stability 
% estimate of a spatial semi-discretization from our previous work~\cite{LiZhou2024A} 
% plays a crucial role in our analysis. We impose relatively mild regularity conditions 
% on the operator $ F $ in \cref{eq:model}. Although our analysis assumes that the 
% initial value $ v $ belongs to $ W_0^{1,\infty}(\mathcal O) \cap W^{2,\infty}(\mathcal O) $, 
% it can be extended to the case that $ v $ possesses lower regularity. Notably,
% our work proposes a novel method for the numerical analysis of nonlinear stochastic
% parabolic equations, offering a new perspective compared to the existing literature
% in the field.

% The outline is not required, but we show an example here.
%The paper is organized as follows. Our main results are in
%\cref{sec:main}, our new algorithm is in \cref{sec:alg}, experimental
%results are in \cref{sec:experiments}, and the conclusions follow in
%\cref{sec:conclusions}.

% \section{Main result}

\section{Mathematical Setup}
\label{sec:math}
For any two Banach spaces \(E_1\) and \(E_2\), we denote by \(\mathcal{L}(E_1, E_2)\) the space of all bounded linear operators from \(E_1\) to \(E_2\), and abbreviate \(\mathcal{L}(E_1, E_1)\) to \(\mathcal{L}(E_1)\). The symbol \(I\) represents the identity map.
Let \(\mathcal{O}\) be a bounded, convex domain in \(\mathbb{R}^3\) of class \(\mathcal{C}^2\) (cf.~Section 11.1 in Chapter 1 of \cite{Yagi2010}). For any \(q \in [1, \infty]\), let \(W_0^{1,q}(\mathcal{O})\) and \(W^{2,q}(\mathcal{O})\) denote the usual Sobolev spaces (see, e.g., \cite{Tartar2007}), and we shall abbreviate \(L^q(\mathcal{O})\) to \(L^q\) for simplicity.
Let \(\Delta\) denote the Laplace operator subject to homogeneous Dirichlet boundary conditions. For \(\alpha \in [0, \infty)\) and \(q \in (1, \infty)\), let \(\dot{H}^{\alpha, q}\) denote the domain of \((- \Delta)^{\alpha/2}\) in \(L^q\), endowed with the norm
\[
\norm{u}_{\dot{H}^{\alpha, q}} := \norm{(-\Delta)^{\alpha/2} u}_{L^q}, \quad \forall u \in \dot{H}^{\alpha, q}.
\]
We use \(\dot{H}^{-\alpha, q'}\) to denote the dual space of \(\dot{H}^{\alpha, q}\), where \(q'\) is the conjugate exponent of \(q\), i.e., \(1/q + 1/q' = 1\).
For any \(q \in (1, \infty)\), it is well-known (see, e.g., Theorem 16.14 of \cite{Yagi2010}) that \(\dot{H}^{2, q}\) coincides with the intersection \(W_0^{1, q}(\mathcal{O}) \cap W^{2, q}(\mathcal{O})\), and is equipped with a norm equivalent to that of \(W^{2, q}(\mathcal{O})\). Moreover, \(\dot{H}^{1, q}\) is identical to \(W_0^{1, q}(\mathcal{O})\), with equivalent norms.

Assume that $ (\Omega, \mathcal F, \mathbb P) $ is a given
complete probability space with a right-continuous filtration $ \mathbb F :=
(\mathcal F_t)_{t \geqslant 0} $, on which a sequence of
independent $ \mathbb F $-adapted Brownian motions
$ (\beta_n)_{n \in \mathbb N} $ are defined. For any $ 0 \leqslant s < t < \infty $ and for any $ n \in \mathbb N $,
the increment $ \beta_n(t) - \beta_n(s) $ is independent of $ \mathcal F_s $.
Let $ H $ be a separable Hilbert
space with inner product $ (\cdot,\cdot)_H $ and an orthonormal basis $ (h_n)_{n \in \mathbb N} $.
Let the $ H $-cylindrical Brownian motion be defined such that for each $ t \in \mathbb R_{+} $,
 $ W_H(t) \in \mathcal L(H, L^2(\Omega)) $ is given by
\[ 
  W_H(t)h := \sum_{n \in \mathbb N} \beta_n(t) (h,h_n)_H,
  \quad \forall h \in H.
\]
We use the symbol \( \mathbb E \) to represent the expectation operator associated with the probability space \( (\Omega, \mathcal F, \mathbb P) \).
For any UMD Banach space \( E \), let \( \gamma(H, E) \) denote the space of all \( \gamma \)-radonifying operators
from \( H \) to \( E \). For any \( p \in [2, \infty) \),
we denote by \( L_\mathbb{F}^p(\Omega \times (0, T); \gamma(H, E)) \)
and $ L_\mathbb F^p(\Omega;L^2(0,T;\gamma(H,E))) $ the spaces
of all \( \mathbb{F} \)-adapted \( \gamma(H, E) \)-valued processes
in \( L^p(\Omega \times (0, T); \gamma(H, E)) \) and \( L^p(\Omega; L^2(0, T; \gamma(H, E))) \),
respectively.
When $ E $ is a Hilbert space, $ \gamma(H,E) $ coincides with the standard Hilbert-Schmidt space
from $ H $ to $E $, and we have the celebrated It\^o's isometry:
\[
  \mathbb E \left\lVert\int_0^T g(t) \, \mathrm{d}W_H(t)\right\rVert_{E}^2
  = \mathbb E\int_0^T \norm{g(t)}_{\gamma(H,E)}^2 \, \mathrm{d}t,
\]
for all $ g \in L_\mathbb F^2(\Omega\times(0,T);\gamma(H,E)) $.
For an in-depth discussion of \( \gamma \)-radonifying operators,
the reader is referred to Chapter 9 of \cite{HytonenWeis2017}.
An extensive treatment of stochastic integrals with respect to $W_H$ in UMD Banach spaces can
be found in \cite{Neerven2007}.

Next, we define the diffusion coefficient $ F $ in \cref{eq:model} as follows.
Let $ (f_n)_{n \in \mathbb N} $ denote a sequence of continuous functions
from $ \overline{\mathcal O} \times \mathbb R $ to $ \mathbb R $ such
that for each $ n \in \mathbb N $,
\begin{equation}
  \label{eq:fn-growth}
  \begin{cases}
  f_n(x,0) = 0,
    & \forall x \in \partial\mathcal O, \\
    \snm{f_n(x,y)} \leqslant C_F(1+\snm{y}),
    & \forall x \in \mathcal O, \, \forall y \in \mathbb R, \\
    \snm{f_n(x,y_1) - f_n(x,y_2)} \leqslant C_F \snm{y_1-y_2},
    & \forall x \in \mathcal O, \, \forall y_1,y_2 \in \mathbb R, \\
    \snm{f_n(x_1,y) - f_n(x_2,y)} \leqslant C_F(1 + \snm{y})\snm{x_1-x_2},
    & \forall x_1,x_2 \in \mathcal O, \, \forall y \in \mathbb R, 
  \end{cases}
\end{equation}
 where $ C_F $ is a positive constant independent of $ n $, $ x $, $ y $,
 $ x_1 $, $ x_2 $, $ y_1 $ and $ y_2 $.
 Let $ (\lambda_n)_{n \in \mathbb N} $ be a sequence of non-negative numbers
 satisfying the condition $ \sum_{n \in \mathbb N} \lambda_n < \infty $.
 Finally, for any $ u \in L^2 $ define 
 \[
   F(u) := \sum_{n=0}^\infty \sqrt\lambda_n h_n \otimes f_n(\cdot,u(\cdot)),
 \]
 where $ h_n \otimes f_n(\cdot,u(\cdot)) \in \gamma(H,L^2) $ is defined by 
 \[
   \big( h_n \otimes f_n(\cdot,u(\cdot)) \big)(w) := (w,h_n)_H f_n(\cdot,u(\cdot)), \quad \forall w \in H.
 \]
 By \cite[Proposition~9.1.9]{HytonenWeis2017} and \eqref{eq:fn-growth},
 a straightforward calculation gives the following growth property:
 \begin{align*}
   \norm{F(u)}_{\gamma(H,L^2)} \leqslant 
   c \big(1 + \norm{u}_{L^2}\big),
   \quad \forall u \in L^2,
 \end{align*}
 where $ c $ is a positive constant independent of $ u $.

 We call a process \( y \in L_\mathbb{F}^6(\Omega \times (0, T); L^6) \) a mild solution to the model problem \cref{eq:model} if it satisfies almost surely the following equality for all \( t \in [0, T] \):
 \[
   y(t) = e^{t\Delta} v + \int_0^t e^{(t-s)\Delta} \big[ y(s) - y^3(s) \big] \, \mathrm{d}s + \int_0^t e^{(t-s)\Delta} F(y(s)) \, \mathrm{d}W_H(s),
 \]
 where \( e^{t\Delta} \), \( t \in [0, T] \), denotes the analytic semigroup generated by the Laplace
 operator \(\Delta\) with homogeneous Dirichlet boundary conditions, and $ v $ is the given
 initial value. While the condition \( y \in L_\mathbb{F}^6(\Omega \times (0, T); L^6) \) can be relaxed 
 (see, e.g., Chapter 6 of \cite{Prato2014}), it suffices for our purposes.
 Following the proof of Theorem 3.1 in \cite{LiZhou2024A}, we establish the following regularity result.
 For related regularity results, we also refer the reader to \cite[Section 3]{LiuQiao2021} and \cite{LiuWei2013}.
 \begin{proposition}
   \label{prop:y-regu}
   Suppose the initial value \( v \) belongs to \( W_0^{1,\infty}(\mathcal{O}) \cap W^{2,\infty}(\mathcal{O}) \). Then, the model problem \eqref{eq:model} admits a unique mild solution. Moreover, for any \( p \in (2, \infty) \), \( q \in [2, \infty) \), and \( \epsilon > 0 \), the solution \( y \) satisfies
   \[
     y \in L_\mathbb{F}^p(\Omega \times (0, T); \dot{H}^{2, q}) \cap L^p(\Omega; C([0, T]; \dot{H}^{2-\epsilon, q})).
   \]
 \end{proposition}

\section{Full Discretization}
\label{sec:main}
Let \(J\) be a positive integer, and set \(t_j := j \tau\) for each \(0 \leqslant j \leqslant J\),
where \(\tau := T/J\) is the time step. Let \(\mathcal{K}_h\) be a conforming and quasi-uniform
triangulation of \(\mathcal{O}\) consisting of three-dimensional simplexes,
and let the spatial mesh size \(h\) denote the maximum diameter of the simplexes in \(\mathcal{K}_h\).
Define
\[
  X_h := \Big\{u_h \in C(\overline{\mathcal{O}}) \,\big|\,
    u_h = 0 \text{ on } \partial\mathcal{O}_h \text{ and outside of } \mathcal{O}_h,
  \text{ and } u_h|_K \in P_r(K) \text{ for each } K \in \mathcal{K}_h \Big\},
\]
where \(C(\overline{\mathcal{O}})\) is the set of all continuous functions
on the closure of \(\mathcal{O}\), \(\mathcal{O}_h\) is the closure of
the union of the simplexes in \(\mathcal{K}_h\), \(r \geqslant 1\)
is an integer, and \(P_r(K)\) is the set of all polynomials on \(K\)
with degree not exceeding \(r\) for each \(K \in \mathcal{K}_h\).
Let \(P_h\) be the \(L^2\)-orthogonal projection operator onto \(X_h\),
and, for any \(u_h \in X_h\), define \(\Delta_hu_h \in X_h\) by
\[
  \int_{\mathcal{O}} (\Delta_hu_h) \cdot v_h \, \mathrm{d}\mu =
  -\int_{\mathcal{O}} \nabla u_h \cdot \nabla v_h
  \, \mathrm{d}\mu, \quad \forall v_h \in X_h,
\]
where \(\mu\) denotes the three-dimensional Lebesgue measure on \(\mathcal{O}\).

This study considers the following full discretization of the model problem \cref{eq:model}:
\begin{equation}
  \label{eq:Y}
  \begin{cases}
    Y_{j+1} - Y_j =
    \tau \left(\Delta_h Y_{j+1} + Y_{j} - P_h Y_{j+1}^3\right)
    + P_h \int_{t_j}^{t_{j+1}} F(Y_j) \, \mathrm{d}W_H(t), \quad 0 \leqslant j < J, \\
    Y_0 = P_h v.
  \end{cases}
\end{equation}
For this full discretization, we establish the following pathwise uniform convergence estimate; the proof is deferred to \cref{sec:proof}.

\begin{theorem}
  \label{thm:convLq}
  Assume that the initial value \( v \) belongs to \( W_0^{1,\infty}(\mathcal{O}) \cap W^{2,\infty}(\mathcal{O}) \),
  and \( \tau \leqslant h^2 \). Let \( y \) be the mild solution of the model problem \cref{eq:model},
  and let \( (Y_j)_{j=0}^J \) be the solution of the full discretization \cref{eq:Y}.
  Then for all \( p \in (2, \infty) \), \( q \in [2, \infty) \), and \( \varepsilon > 0 \),
  the following pathwise uniform convergence estimate holds:
  \begin{equation}
    \label{eq:convLpLq-full}
    \normB{
      \max_{1 \leqslant j \leqslant J}
      \norm{Y_j - y(t_j)}_{L^q}
    }_{L^p(\Omega)} \leqslant 
    c \big( h^{2-\varepsilon} + \tau^{1/2} \big),
  \end{equation}
  where \( c \) denotes a positive constant independent of the spatial mesh size \( h \)
  and the time step \( \tau \).
\end{theorem}

\begin{remark}
  For \cref{thm:convLq}, we have the following comments.
  \begin{enumerate}
    \item The condition \( \tau \leqslant h^2 \) can be relaxed by re-examining and refining the analysis in Subsection~\ref{sssec:love}.
    \item Assume that the sequence \( (f_n)_{n \in \mathbb{N}} \), in addition to satisfying \eqref{eq:fn-growth}, consists of twice continuously differentiable functions and satisfies the following growth condition for each \( n \in \mathbb{N} \):
      \begin{align*}
        \left| \nabla_x^2 f_n(x, y) \right|
        + \left| \nabla_x \nabla_y f_n(x, y) \right|
        + \left| \nabla_y^2 f_n(x, y) \right| \leqslant C_F (1 + |y|),
        \quad \forall x \in \mathcal{O}, \, \forall y \in \mathbb{R}^3,
      \end{align*}
      where \( C_F \) is a positive constant independent of \( n \). Under the condition that the initial value \( v \) belongs to \( W_0^{1,\infty}(\mathcal{O}) \cap W^{2,\infty}(\mathcal{O}) \), the spatial accuracy \( O(h^{2-\epsilon}) \) in \cref{eq:convLpLq-full} can be improved to \( O(h^2) \) for all \( p \in (2, \infty) \) and \( q \in [2, \infty) \).
    \item The authors in~\cite{LiZhou2024A} derived a strong convergence rate for a spatial semi-discretization of \cref{eq:model} with a rough initial value \( v \in L^\infty \). However, to the best of our knowledge, there is currently no strong convergence estimate for fully discrete approximations of the three-dimensional stochastic Allen-Cahn equation with a rough initial value. It would be of interest to establish similar error estimates as \cref{eq:convLpLq-full} for fully discrete approximations with rough initial values.
  \end{enumerate}
\end{remark}

       %      % \item For the discretization considered in this paper,
       %      %   Majee and Prohl~\cite{Prohl2018} obtained the convergence rate
       %      %   $ O(h + \tau^{1/2-\epsilon}) $, and \cite{LiuQiao2021} obtained 
       %      %   converence rate $ O(h^2+\tau^{1/2}) $, with differnt regularition
       %      %   conditions on the diffusion term.
       %      % \item For the discretization considered in this paper,
       %      %   Majee and Prohl~\cite{Prohl2018} achieved a convergence rate
       %      %   of $O(h + \tau^{1/2-\epsilon})$, while \cite{LiuQiao2021} obtained
       %      %   a convergence rate of $O(h^2+\tau^{1/2})$. These results were achieved
       %      %   under different regularity conditions on the diffusion term. 
       %      % \item In comparison to prior research, Majee and Prohl~\cite{Prohl2018}
       %      %   achieved a convergence rate of $O(h + \tau^{1/2-\epsilon})$, while
       %      %   Liu and Qiao~\cite{LiuQiao2021} obtained a rate of
       %      %   $O(h^2 + \tau^{1/2})$. These outcomes were realized under
       %      %   distinct regularity conditions on the diffusion term.

% Theorem

%Theorem

 \section{Proofs}
 \label{sec:proof}
In this section, we adopt the following notational conventions:
\begin{itemize}
  \item The notation \( \langle \cdot, \cdot \rangle \) denotes the integral over the domain \( \mathcal{O} \).
  \item For \( \alpha \in \mathbb{R} \) and \( q \in (1, \infty) \),
    \( \dot{H}_h^{\alpha, q} \) represents the Banach space \( X_h \) equipped
    with the norm
    \[
      \norm{u_h}_{\dot{H}_h^{\alpha, q}} := \norm{(-\Delta_h)^{\alpha/2} u_h}_{L^q}, \quad \forall u_h \in X_h.
    \]
  \item The symbol \( c \) denotes a generic positive constant,
    which is independent of the spatial mesh size \( h \) and the time step \( \tau \).
    Its specific value may depend on the regularity parameters associated with
    the triangulation \( \mathcal{K}_h \),
    and may change from line to line.
\end{itemize}

The subsequent structure of this section is outlined as follows.
In Subsection \ref{ssec:pre}, we introduce several preliminary lemmas that
form the foundation for our subsequent analysis.
Subsequently, Subsection \ref{subsec:stability} is dedicated to establishing a
stability estimate for a discrete stochastic convolution,
which is of independent interest and also crucial for
the proof of \cref{thm:convLq}.
Lastly, Subsection \ref{subsec:convLq} presents a detailed proof of \cref{thm:convLq}.

\subsection{Preliminary lemmas}
\label{ssec:pre}

According to the theoretical results in \cite{Douglas1975},
the projection operator $ P_h $ has the following well-known stability property.
\begin{lemma} 
  \label{lem:Ph-stab}
  For any $ \alpha \in [0,2] $ and $q \in (1,\infty) $,
  $ \norm{P_h}_{\mathcal L(\dot H^{\alpha,q}, \, \dot H_h^{\alpha,q})} $
  is uniformly bounded with respect to the spatial mesh size $ h $.
\end{lemma}

Referring to the established estimates for the resolvents of $ \Delta_h $
as established in \cite{Bakaev2002}, a standard computation
(for which the relevant techniques can be found in  \cite[Section 7.7 in Chapter 2]{Yagi2010})
yields the following well-recognized inequalities.
\begin{lemma}
  \label{lem:Deltah}
  Let $ q \in (1,\infty) $ and $ 0 \leqslant \alpha \leqslant \beta < \infty $.
  Then the following inequalities hold:
  \begin{align}
    & \norm{e^{t\Delta_h}}_{\mathcal L(\dot H_h^{\alpha,q}, \dot H_h^{\beta,q})}
    \leqslant ct^{(\alpha-\beta)/2}, \quad \forall t \in (0,T]; 
    \label{eq:etDeltah} \\
    & \norm{I-e^{t\Delta_h}}_{\mathcal L(\dot H_h^{\beta,q}, \dot H_h^{\alpha,q})}
    \leqslant ct^{(\beta-\alpha)/2}, \quad \forall t \in [0,T];
    \label{eq:I-etDeltah} \\
    & \norm{(I-\tau\Delta_h)^{-1} - I}_{\mathcal L(\dot H_h^{\beta,q}, \dot H_h^{\alpha,q})}
    \leqslant c\tau^{(\beta-\alpha)/2};
    \label{eq:basic} \\
    & \norm{(I-\tau\Delta_h)^{-m}}_{
      \mathcal L(\dot H_h^{\alpha,q}, \dot H_h^{\beta,q})
    } \leqslant \frac{c}{(m\tau)^{(\beta-\alpha)/2}},
    \quad  \, \forall m \in \mathbb N_{>0}. 
    \label{eq:resolvent} 
  \end{align}
\end{lemma}

We have the following fundamental estimate for the stochastic integrals;
see, e.g., \cite[pp.~792]{Neerven2012}. 
\begin{lemma}
  \label{lem:stoch}
  Let $ p \in (1,\infty) $ and $ q \in [2,\infty) $.
  There exists a positive constant $ c $ such that
  for any $ g \in L_\mathbb F^p(\Omega;L^2(0,T;\gamma(H,L^q))) $,
  the following inequality holds:
  \[
    \normB{\int_0^T g(t)\, \mathrm{d}W_H(t)}_{L^p(\Omega;L^q)}
    \leqslant c \norm{g}_{L^p(\Omega;L^2(0,T;\gamma(H;L^q)))}.
  \]
\end{lemma}

% \begin{lemma}
%   \label{lem:F-interp}
%   Let $ q \in [2,\infty) $. Then the following holds:
%   \begin{align}
%       & \norm{F(v)}_{\gamma(H,L^q)} \leqslant
%       c(1 + \norm{v}_{L^q}),
%       & \quad \forall v \in L^q,
%       \label{eq:F-growth} \\
%       & \norm{F(v)}_{\gamma(H,\dot H^{1,q})} \leqslant
%       c(1 + \norm{v}_{\dot H^{1,q}}), 
%       & \forall v \in \dot H^{1,q},
%       \label{eq:F-growth-H1} \\ 
%       & \norm{F(u) - F(v)}_{\gamma(H,L^q)} \leqslant
%       c \norm{u-v}_{L^q},
%       & \forall u,v \in L^q.
%       \label{eq:F-lips} 
%   \end{align}
% \end{lemma}
% \begin{proof}
% These properties follow from straightforward calculations based on the definition of \( F \);
% see \cite[Lemma~4.5]{LiZhou2024A}.
% \end{proof}

\begin{lemma}
\label{lem:F-interp}
Suppose \( q \in [2,\infty) \). The following inequalities hold:
\begin{align}
& \norm{F(u)}_{\gamma(H,L^q)} \leqslant c(1 + \norm{u}_{L^q}) && \text{for all } u \in L^q, \label{eq:F-growth} \\
& \norm{F(u)}_{\gamma(H,\dot H^{1,q})} \leqslant c(1 + \norm{u}_{\dot H^{1,q}}) && \text{for all } u \in \dot H^{1,q}, \label{eq:F-growth-H1} \\
& \norm{F(u_1) - F(u_2)}_{\gamma(H,L^q)} \leqslant c \norm{u_1-u_2}_{L^q} && \text{for all } u_1,u_2 \in L^q. \label{eq:F-lips}
\end{align}
\end{lemma}
\begin{proof}
These properties are derived from straightforward calculations based on the definition of \( F \). For a detailed derivation, see Lemma~4.5 in \cite{LiZhou2024A}.
\end{proof}

We have the following discrete deterministic maximal $ L^p $-regularity
estimate; see \cite{Kemmochi2016,Libuyang2015}.
\begin{lemma}
  \label{lem:DMLP}
  Let $ p,q \in (1,\infty) $. Assume that \((Z_j)_{j=0}^J \) is
  the solution of the discretization
  \begin{numcases}{}
    Z_{j+1} - Z_j = \tau \Delta_h Z_{j+1} +
    P_h\int_{t_j}^{t_{j+1}} g(t) \, \mathrm{d}t, \quad 0 \leqslant j < J, \notag \\
    Z_0 = 0, \notag
  \end{numcases}
  where $ g \in L^p(0,T;L^q) $. Then the following stability estimate holds:
  \[
    \Big(
      \tau\sum_{j=1}^J \norm{Z_j}_{\dot H_h^{2,q}}^p 
    \Big)^{1/p} \leqslant
    c \norm{g}_{L^p(0,T;L^q)}.
  \]
\end{lemma}

Finally, we present some standard properties of the space $\dot H_h^{\alpha,q}$, where $\alpha \in \mathbb{R}$ and $q \in (1,\infty)$.
For any $q \in (1,\infty)$, there exist two positive constants $c_0$ and $c_1$, independent of $h$, such that for all $u_h \in X_h$, the following inequality holds:
\[
c_0\norm{u_h}_{\dot H^{1,q}} \leqslant \norm{u_h}_{\dot H_h^{1,q}} \leqslant c_1\norm{u_h}_{\dot H^{1,q}}.
\]
For any $\alpha \in (0,2)$ and $q \in (1,\infty)$, there exist two positive constants $c_2$ and $c_3$, independent of $h$, such that for all $u_h \in X_h$, the following inequality holds:
\[
c_2\norm{u_h}_{\dot H_h^{\alpha,q}} \leqslant \norm{u_h}_{[\dot H_h^{0,q}, \dot H_h^{2,q}]_{\alpha/2}} \leqslant c_3\norm{u_h}_{\dot H_h^{\alpha,q}},
\]
where $[\dot H_h^{0,q}, \dot H_h^{2,q}]_{\alpha/2}$ is the complex interpolation space between $\dot H_h^{0,q}$ and $\dot H_h^{2,q}$ (see \cite[Section 5 in Chapter 1]{Yagi2010}).
Furthermore, for any $\beta \in [3/2, \infty)$ and $q \in [2, \infty)$,
using the well-known estimate (see, e.g., \cite[Section 8.5]{Brenner2008}),
\[
\norm{\Delta^{-1} - \Delta_h^{-1}P_h}_{\mathcal{L}(L^\beta)} \leqslant ch^2,
\]
combined with the continuous embedding of $\dot{H}^{2, \beta}$ into $L^q$ and the inverse estimate, it follows that the continuous embedding of $\dot{H}_h^{2, \beta}$ into $L^q$ is independent of $h$.
These results will be used throughout the subsequent analysis without further explicit reference.

\subsection{A stability estimate for a discrete stochastic convolution}
\label{subsec:stability}

The objective of this subsection is to establish the stability estimate
\begin{equation}
  \label{eq:DMINEQ}
  \bigg(
    \mathbb{E}\Big[
      \max_{1 \leqslant j \leqslant J} \norm{Z_j}_{L^{q}}^p
    \Big]
  \bigg)^{1/p} \leqslant
  c \norm{g_h}_{L^{p}(\Omega \times (0,T); \gamma(H, \dot{H}_h^{0,q}))}
\end{equation}
for the discrete stochastic convolution defined by
\begin{equation}
  \label{eq:Zj}
  Z_j = \sum_{k=0}^{j-1} \int_{t_k}^{t_{k+1}}
  (I - \tau \Delta_h)^{k-j} g_h(t) \, \mathrm{d}W_H(t), \quad 1 \leqslant j \leqslant J,
\end{equation}
where \( g_h \) belongs to \( L_\mathbb{F}^p(\Omega \times (0,T); \gamma(H, \dot{H}_h^{0,q})) \),
with parameters \( p \in (2, \infty) \) and \( q \in [2, \infty) \).
This stability estimate is crucial for the derivation of
pathwise uniform convergence rates for the numerical approximations of a
broad class of nonlinear stochastic parabolic equations.

Within the framework of Hilbert spaces, we note that
Gy\"ongy and Millet \cite[Theorem 2.6]{Gyongy2009} have
established the following stability estimate:
\[
\left( \mathbb{E} \left[ \max_{1 \leqslant j \leqslant J} \| Z_j \|_{L^2}^2 \right] \right)^{1/2} \leqslant c \left\| g \right\|_{L^2(\Omega \times (0,T); \gamma(H, L^2))}.
\]
Moreover, under the Banach space setting, a significant advancement was made by
van Neerven and Veraar \cite[Proposition 5.4]{Neerven2022}, 
who demonstrated the following sharp estimate for a discrete stochastic convolution
without spatial discretization:
\[
  \left( \mathbb{E} \left[ \max_{1 \leqslant j \leqslant J}
      \Big\| \sum_{k=0}^{j-1} \int_{t_k}^{t_{k+1}} (I - \tau\Delta)^{k-j} g(t)
  \, \mathrm{d}W_H(t) \Big\|_{L^q}^p \right] \right)^{1/p}
  \leqslant c \left\| g \right\|_{L^p(\Omega; L^2(0,T; \gamma(H, L^q)))},
\]
holds for all $ g \in L_\mathbb{F}^p(\Omega; L^2(0,T; \gamma(H, L^q))) $
with $ p,q \in [2,\infty) $. The derivation of this estimate relies on
the contraction property of the operator $ (I-\tau\Delta)^{-1} $ on
$ L^q $ for all $ q \in [2,\infty) $. However, it is important to note
that $ (I-\tau\Delta_h)^{-1} $ fails to possess 
the contraction property on $ \dot H_h^{0,q} $ for any $ q \in (2,\infty) $.
Hence, Proposition~5.4 from \cite{Neerven2012} can not be directly applied to
bound the discrete stochastic convolution defined by \cref{eq:Zj}.

To address the aforementioned challenge, we employ the discrete stochastic maximal $ L^p $-regularity
estimate for 
\[
   \tau \sum_{j=1}^J \mathbb E \norm{Z_j}_{\dot H_h^{1,q}}^p,
   \quad p \in (2,\infty), \, q \in [2,\infty),
\]
coupled with the approximation property of $ (I-\tau\Delta_h)^{-1} $ to the identity operator
$ I $ as presented in \cref{eq:basic}.
We formally state the discrete stochastic maximal $ L^p $-regularity estimate
in the following proposition.
\begin{proposition}
  \label{pro:DSMLP}
  Let $ p \in (2,\infty) $ and $ q \in [2,\infty) $.
  Assume that $ (Z_j)_{j=1}^J $ is defined by \cref{eq:Zj},
  where $ g_h \in L_\mathbb F^p(\Omega\times(0,T); \gamma(H,\dot H_h^{0,q})) $.
  Then the following discrete stochastic
  maximal $ L^p $-regularity estimate holds:
  \[
    \Bigg(
      \mathbb E\bigg[\tau \sum_{j=1}^J \norm{Z_j}_{\dot H_h^{1,q}}^p \bigg]
    \Bigg)^{1/p} \leqslant
    c \norm{g_h}_{L^p(\Omega\times(0,T);\gamma(H,\dot H_h^{0,q}))}.
  \]
\end{proposition}
\begin{proof}
  Note that \cite[Theorem~3.1]{LiLpSpatail2023} has demonstrated that the
  $H^\infty$-calculus of $-\Delta_h$ is uniformly bounded with respect to
  the spatial mesh size $h$. Subsequently, applying \cite[Theorem~3.2]{LiXieLpTime2023},
  we establish the desired discrete stochastic maximal $L^p$-regularity estimate.
\end{proof}

Now we present the main result of this subsection in the form of the following proposition.
\begin{proposition}
  \label{prop:DMINEQ}
  Let $ p \in (2,\infty) $ and $ q \in [2,\infty) $.
  Suppose that $ g_h \in L_\mathbb F^{p}(\Omega\times(0,T);\gamma(H,\dot H_h^{0,q})) $,
  and let $ (Z_j)_{j=1}^J $ be defined by \cref{eq:Zj}.
  Then the stability estimate \cref{eq:DMINEQ} holds.
\end{proposition}
\begin{proof}
  Let $ Z_0 := 0 $ and, for any $t \in [0,T] $, define
  \[
    \Psi(t) := (I-\tau\Delta_h)^{-1}g_h(t), \quad
    \mathcal G(t) := p\sum_{j=0}^{J-1}
    \mathbbm{1}_{[t_j,t_{j+1})}(t)
    \norm{Z_j}_{L^q}^{p-q}
    \left\langle\snm{Z_j}^{q-2} Z_j, \, \Psi(t)\right\rangle,
  \]
  where $ \mathbbm{1}_{[t_j,t_{j+1})} $ is the indicator function for
  the time interval $ [t_j,t_{j+1}) $.
  For any $ 0 \leqslant j < J $, we also define
  \[
    \delta_j := (I-\tau\Delta_h)^{-1}Z_j - Z_j, 
    \quad
    M_j := \delta_j + \int_{t_j}^{t_{j+1}} \Psi(t) \, \mathrm{d}W_H(t).
  \]
  It can be verified that
  \begin{align*}
    Z_{j+1} = Z_j + M_j \quad \text{for all $0 \leqslant j < J$}.
  \end{align*}
  The rest of the proof is divided into the following five steps.

  \textbf{Step 1}. Let us prove
  \begin{equation}
    \label{eq:729}
    \mathbb E \Big[
      \max_{0 \leqslant j \leqslant J} \norm{Z_j}_{L^q}^p
    \Big] \leqslant
    I_1 + I_2 + I_3 + I_4,
  \end{equation}
  where
  \[
    \hskip -16.5em I_1 := p \mathbb E \bigg[
      \sum_{j=0}^{J-1}
      \norm{Z_j}_{L^q}^{p-q}
      \left|\left\langle\snm{Z_j}^{q-2}Z_j, \, \delta_j\right\rangle\right|
    \bigg], 
    \]
    \[
      \hskip -19.5em I_2 := \mathbb E \bigg[
          \sup_{t \in [0,T]} 
          \snmB{
            \int_0^t \mathcal G(s) \, \mathrm{d}W_H(s)
          }
        \bigg], 
      \]
      \[
        I_3 :=  p|p-q| \mathbb E\bigg[
          \sum_{j=0}^{J-1} \int_0^1 \norm{Z_j+\theta M_j}_{L^q}^{p-2q}
            \dualb{\snm{Z_j+\theta M_j}^{q-2}(Z_j+\theta M_j), \, M_j}^2
            \, \mathrm{d}\theta
  \bigg],
\]
\[
  \hskip -6em
  I_4 := p(q-1) \mathbb E\bigg[ \sum_{j=0}^{J-1} 
      \int_0^1 
      \norm{Z_j+\theta M_j}_{L^q}^{p-q} \dualb{
        \snm{Z_j+\theta M_j}^{q-2}, \, M_j^2
    } \, \mathrm{d}\theta
  \bigg].
\]
  Fix any $ 0 \leqslant j < J $.
  Define the function $ g: [0,1] \to \mathbb R $ by
  \[
    g(\theta) := \normb{ Z_j + \theta M_j }_{L^q}^p,
    \quad \theta \in [0,1].
  \]
  A straightforward computation yields
  \begin{align*}
      & g'(0) = p \norm{Z_j}_{L^q}^{p-q} 
      \dualb{\snm{Z_j}^{q-2}Z_j, \, M_j} \\
      & \qquad = p \norm{Z_j}_{L^q}^{p-q} 
      \dualb{\snm{Z_j}^{q-2}Z_j, \, \delta_j} +
      \int_{t_j}^{t_{j+1}} \mathcal G(t) \, \mathrm{d}W_H(t)
  \end{align*}
  and
  \begin{align*}
      & g''(\theta) 
      = p(p-q) \normb{ Z_j + \theta M_j }_{L^q}^{p-2q} 
      \dualb{
        \snm{Z_j+\theta M_j}^{q-2}(Z_j+\theta M_j),
        \, M_j
      }^2 + {} \\
      & \qquad\qquad\quad p(q-1) \normb{ Z_j + \theta M_j }_{L^q}^{p-q} \dualb{
        \snm{Z_j+\theta M_j}^{q-2}, \, M_j^2
      }, \quad \forall\theta \in [0,1].
  \end{align*}
  Hence, by $ g(1) = g(0)+ g'(0) +
  \int_0^1 (1-\theta) g''(\theta) \mathrm{d}\theta $ and the fact that
  $ g(1) = \norm{Z_{j+1}}_{L^q}^p $ and $ g(0) = \norm{Z_j}_{L^q}^p $, we obtain
  \begin{align*}
    \norm{Z_{j+1}}_{L^q}^p 
      &\leqslant \norm{Z_j}_{L^q}^p + 
      p \norm{Z_j}_{L^q}^{p-q}
      \dualb{\snm{Z_j}^{q-2}Z_j, \, \delta_j} +
      \int_{t_j}^{t_{j+1}} \mathcal G(t) \, \mathrm{d}W_H(t)  \\
      &\quad {} +  p|p-q| \int_0^1
      \norm{Z_j+\theta M_j}_{L^q}^{p-2q}
      \dualb{\snm{Z_j+\theta M_j}^{q-2}(Z_j+\theta M_j), \, M_j}^2 \, \mathrm{d}\theta \\
      & \quad {}+p(q-1) \int_0^1 
      \norm{Z_j+\theta M_j}_{L^q}^{p-q} \dualb{
        \snm{Z_j+\theta M_j}^{q-2}, \, M_j^2
      } \, \mathrm{d}\theta.
  \end{align*}
Given that \(Z_0 = 0\) and the preceding inequality holds for all \(0 \leqslant j < J\),
we deduce the following bound for the supremum of the sequence \((Z_j)_{j=0}^J\):
\begin{align*}
& \max_{0 \leqslant j \leqslant J} \left\|Z_j\right\|_{L^q}^p
\leqslant p \sum_{j=0}^{J-1}
\left\|Z_j\right\|_{L^q}^{p-q}
\left|\dualb{ \left|Z_j\right|^{q-2}Z_j,\, \delta_j }\right| 
+ \sup_{t \in [0,T]} \left|\int_{0}^{t} \mathcal{G}(s) \, \mathrm{d}W_H(s)\right| \\
&\qquad\qquad {} + p|p-q| \sum_{j=0}^{J-1}
\int_0^1 \left\|Z_j+\theta M_j\right\|_{L^q}^{p-2q}
\dualb{ \left|Z_j+\theta M_j\right|^{q-2}(Z_j+\theta M_j), \, M_j }^2 \, \mathrm{d}\theta \\
&\qquad\qquad {} + p(q-1) \sum_{j=0}^{J-1} \int_0^1
\left\|Z_j+\theta M_j\right\|_{L^q}^{p-q}
\dualb{ \left|Z_j+\theta M_j\right|^{q-2}, \, M_j^2 } \, \mathrm{d}\theta.
\end{align*}
By taking expectations on both sides of the inequality, we arrive at
the desired inequality \eqref{eq:729}.

\textbf{Step 2}.
From the definition of \(I_1\), it follows that
\begin{align*}
  I_1 & \leqslant p \mathbb{E}\bigg[
    \sum_{j=0}^{J-1}
    \norm{Z_j}_{L^q}^{p-q}
    \normb{\snm{Z_j}^{q-2}Z_j}_{\dot{H}^{1,q/(q-1)}}
    \norm{\delta_j}_{\dot{H}^{-1,q}}
  \bigg] \\
  & \leqslant c \mathbb{E}\bigg[
    \sum_{j=0}^{J-1} \norm{Z_j}_{L^q}^{p-q}
    \normb{\snm{Z_j}^{q-2}Z_j}_{\dot{H}^{1,q/(q-1)}}
    \norm{\delta_j}_{\dot{H}_h^{-1,q}}
  \bigg].
\end{align*}
For any \(0 \leqslant j < J\), we observe that
\begin{align*}
  \|\delta_j\|_{\dot{H}_h^{-1,q}} 
  &\leqslant
  \|(I-\tau\Delta_h)^{-1} - I\|_{\mathcal{L}(\dot{H}_h^{1,q}, \dot{H}_h^{-1,q})}
  \|Z_j\|_{\dot{H}_h^{1,q}} \\
  &=
  \|(I-\tau\Delta_h)^{-1} - I\|_{\mathcal{L}(\dot{H}_h^{2,q}, \dot{H}_h^{0,q})}
  \|Z_j\|_{\dot{H}_h^{1,q}} \\
  &\leqslant c\tau \|Z_j\|_{\dot{H}_h^{1,q}},
\end{align*}
where the last inequality is justified by \eqref{eq:basic} with \(\alpha = 0\) and \(\beta = 2\).
Furthermore,
\begin{align*}
  \|\,|Z_j|^{q-2}Z_j\|_{\dot{H}^{1,q/(q-1)}} 
  &\leqslant
  c\|\nabla(|Z_j|^{q-2}Z_j)\|_{L^{q/(q-1)}} \\
  &\leqslant
  c \|Z_j\|_{L^q}^{q-2}
  \|\nabla Z_j\|_{L^q} \quad\text{(by H\"older's inequality)} \\
  &\leqslant c \|Z_j\|_{\dot{H}_h^{1,q}}^{q-1}.
\end{align*}
Consequently, putting the above estimates together, we obtain
\begin{equation}
  \label{eq:730}
  I_1 \leqslant 
  c \mathbb{E}\bigg[ \tau \sum_{j=0}^{J-1} \norm{Z_j}_{\dot{H}_h^{1,q}}^p
  \bigg].
\end{equation}

\textbf{Step 3}. 
By the Burkholder-Davis-Gundy inequality and Hölder's inequality, we obtain
\begin{align*}
  I_2
  & \leqslant
  c \mathbb{E} \Big[
    \int_0^{T} \norm{\mathcal{G}(t)}_{\gamma(H, \mathbb{R})}^2 \, \mathrm{d}t
  \Big]^{1/2} \\
  &= c \mathbb{E} \Big[
    \sum_{j=0}^{J-1} \int_{t_j}^{t_{j+1}}
    \normB{
      \norm{Z_j}_{L^q}^{p-q} \dualb{\snm{Z_j}^{q-2}Z_j, \, \Psi(t)}
    }_{\gamma(H, \mathbb{R})}^2 \, \mathrm{d}t
  \Big]^{1/2} \\
  &\leqslant
  c \mathbb{E} \Big[
    \sum_{j=0}^{J-1} \int_{t_j}^{t_{j+1}}
    \norm{Z_j}_{L^q}^{2p-2q}
    \normb{\snm{Z_j}^{q-2}Z_j}_{L^{q/(q-1)}}^2
    \norm{\Psi(t)}_{\gamma(H, L^q)}^2 \, \mathrm{d}t
  \Big]^{1/2} \\
  &= c \mathbb{E} \Big[
    \sum_{j=0}^{J-1} \int_{t_j}^{t_{j+1}}
    \norm{Z_j}_{L^q}^{2p-2}
    \norm{\Psi(t)}_{\gamma(H, L^q)}^2 \, \mathrm{d}t
  \Big]^{1/2}.
\end{align*}
It follows that
\begin{align*} 
  I_2 \leqslant c \mathbb{E} \Big[
    \max_{0 \leqslant j \leqslant J} \norm{Z_j}_{L^q}^{p-1}
    \norm{\Psi}_{L^2(0,T; \gamma(H, L^q))}
  \Big].
\end{align*}
By Young's inequality, we then obtain
\begin{equation}
  \label{eq:731}
  I_2 \leqslant 
  \frac{1}{2} \mathbb{E} \Big[ \max_{0 \leqslant j \leqslant J}
  \norm{Z_j}_{L^q}^p \Big] + c \norm{\Psi}_{L^p(\Omega; L^2(0,T;\gamma(H, L^q)))}^p.
\end{equation}

    \textbf{Step 4}. 
We employ Hölder's inequality and Young's inequality as follows:
    \begin{align*}
      I_3 + I_4 
      & \leqslant c \mathbb E\bigg[
        \sum_{j=0}^{J-1} 
        \int_0^1 \norm{Z_j+\theta M_j}_{L^q}^{p-2} \norm{M_j}_{L^q}^2 
        \, \mathrm{d}\theta
      \bigg]\\
      &\leqslant c\mathbb E\bigg[
        \sum_{j=0}^{J-1} \mathbb E \Big[
          \int_0^1 
          \tau\norm{Z_j+\theta M_j}_{L^q}^p +
          \tau^{1-p/2}\norm{M_j}_{L^q}^p \, \mathrm{d}\theta
        \bigg] \\
      &\leqslant c \mathbb E\bigg[
        \tau\sum_{j=0}^{J-1} \norm{Z_j}_{L^q}^p
      \bigg] + c \tau^{1-p/2} \mathbb E\bigg[ \sum_{j=0}^{J-1} \norm{M_j}_{L^q}^p \bigg].
    \end{align*}
   Moreover,
    \begin{align*}
      \mathbb E\bigg[\sum_{j=0}^{J-1} \norm{M_j}_{L^q}^p \bigg]
    & \stackrel{\mathrm{(i)}}{\leqslant}
    c \mathbb E\bigg[ \sum_{j=0}^{J-1} \norm{\delta_j}_{L^q}^p \bigg]
    + c \tau^{p/2-1} \sum_{j=0}^{J-1}
    \int_{t_j}^{t_{j+1}} \norm{\Psi(t)}_{L^p(\Omega;\gamma(H,L^q))}^p \, \mathrm{d}t \\
    & \stackrel{\mathrm{(ii)}}{\leqslant}
    c \tau^{p/2} \mathbb E\bigg[  \sum_{j=0}^{J-1} \norm{Z_j}_{\dot H_h^{1,q}}^p \bigg] +
    c\tau^{p/2-1} \norm{\Psi}_{L^p(\Omega\times(0,T);\gamma(H,L^q))}^p, 
    \end{align*}
where in the first inequality (i), we invoke the definition \( M_j := \delta_j + \int_{t_j}^{t_{j+1}} \Psi(t) \, \mathrm{d}W_H(t) \),
Lemma \ref{lem:stoch}, and Hölder's inequality;
and in the second inequality (ii), we make use of the definition \( \delta_j := (I-\tau\Delta_h)^{-1}Z_j - Z_j \),
along with the resolvent estimate from Lemma \ref{eq:basic} with \( \alpha=0 \) and \( \beta=1 \).
By combining these estimates, we obtain
    \begin{equation}
      \label{eq:732}
      I_3 + I_4 \leqslant 
      c \mathbb E\bigg[ \tau \sum_{j=0}^{J-1} 
        \norm{Z_j}_{\dot H_h^{1,q}}^p 
      \bigg] +
      c \norm{\Psi}_{L^p(\Omega\times(0,T);\gamma(H,L^q))}^p.
    \end{equation}

    \textbf{Step 5}. Combining \cref{eq:729,eq:730,eq:731,eq:732}, and noting that $ Z_0 = 0 $,
    we deduce that
    \[
      \mathbb{E}\left[\max_{1 \leqslant j \leqslant J} \norm{Z_j}_{L^q}^p\right]
      \leqslant c\mathbb E\bigg[\tau\sum_{j=1}^{J-1} \norm{Z_j}_{\dot{H}_h^{1,q}}^p
      \bigg] +
      c\norm{\Psi}_{L^p(\Omega\times(0,T);\gamma(H,L^q))}^p.
    \]
Applying \cref{pro:DSMLP}, we infer that
\[
  \mathbb E\bigg[\tau\sum_{j=1}^{J-1} \norm{Z_j}_{\dot{H}_h^{1,q}}^p\bigg]
  \leqslant c\norm{g_h}_{L^p(\Omega\times(0,T);\gamma(H,\dot H_h^{0,q}))}^p.
\]
Moreover, using inequality \eqref{eq:resolvent} with parameters
\(\alpha = \beta = 0\), we establish
\[
\norm{\Psi}_{L^p(\Omega\times(0,T);\gamma(H,L^q))}
\leqslant c\norm{g_h}_{L^p(\Omega\times(0,T);\gamma(H,\dot H_h^{0,q}))}.
\]
Aggregating these estimates, we conclude the desired stability estimate \eqref{eq:DMINEQ}.
This completes the proof.
\end{proof}

\begin{remark}
  Considering the maximal inequality established in \cite[Theorem~1.1]{Neerven2022},
  \[
  \mathbb{E} \bigg[
  \sup_{t \in [0,T]} \left\lVert \int_0^t e^{(t-s)\Delta} g(s) \, \mathrm{d}W_H(s) \right\rVert_{L^q}^p
  \bigg] \leqslant c \mathbb{E} \left( \int_0^T \lVert g(t) \rVert_{\gamma(H,L^q)}^2 \, \mathrm{d}t \right)^{p/2},
  \]
  for all \( g \in L_\mathbb{F}^p(\Omega; L^2(0,T; \gamma(H,L^q))) \) with \( p, q \in [2, \infty) \),
  the stability estimate provided in Proposition \ref{prop:DMINEQ} is suboptimal.
  While Proposition \ref{prop:DMINEQ} is adequate for the purposes of the current investigation,
  it is an interesting question to derive a sharper stability estimate than the one given by \eqref{eq:DMINEQ}.
\end{remark}

% \begin{remark}
%   Here, we provide another stability estimate for the sequence $(Z_j)_{j=1}^J $ defined 
%   by \cref{eq:Zj} as follows: for any $ $
%   \begin{align*}
%     \Bigg(
%       \mathbb E\bigg[
%       \max_{1 \leqslant j \leqslant J} \norm{Z_j}_{L^q}^p
%     \bigg]
%   \Bigg)^{1/p} \leqslant c\norm{g_h}_{L^p(\Omega\times(0,T);\dot H_h^{2/p-1+\epsilon,q})}
%   \end{align*}
% \end{remark}
% \begin{proof}
%   Define 
%   \[
%     z_h(t) := \int_0^t e^{(t-s)\Delta_h} g_h(s) \, \mathrm{d}W_H(s), \quad t \in [0,T].
%   \]
%   We have
%   \[
%     \norm{z_h}_{L^p(\Omega;C([0,T];\dot H_h^{1-2/p-\epsilon,q}))}
%   \]
%   is uniformly bounded with respect to $ h $
% \end{proof}

\subsection{Proof of \texorpdfstring{\cref{thm:convLq}}{}}
\label{subsec:convLq}
In this subsection, we always assume that the conditions of Theorem \ref{thm:convLq} are satisfied. Let \( y_h \) be the strong solution of the following spatial semi-discretization:
\begin{equation}
  \label{eq:yh}
  \begin{cases}
    \mathrm{d}y_h(t) = (\Delta_h y_h + y_h - P_h y_h^3)(t) \, \mathrm{d}t +
    P_h F(y_h(t)) \, \mathrm{d}W_H(t), & t \in [0,T], \\
    y_h(0) = P_h v.
  \end{cases}
\end{equation}
Li and Zhou~\cite[Theorem~3.3]{LiZhou2024A} have established the error estimate
\begin{equation}
  \label{eq:y-yh}
  \norm{y - y_h}_{L^p(\Omega; C([0,T]; L^q))} \leqslant 
  ch^{2-\varepsilon}
\end{equation}
for all \( p \in (2, \infty) \), \( q \in [2, \infty) \), and \( \varepsilon > 0 \).
Therefore, to complete the proof of Theorem \ref{thm:convLq}, it remains to prove
\begin{equation}
  \label{eq:convLpLq}
  \normB{
    \max_{1 \leqslant j \leqslant J}
    \norm{y_h(t_j) - Y_j}_{L^{q}}
  }_{L^{p}(\Omega)} \leqslant c \tau^{1/2},
  \quad \forall p \in (2, \infty), \, \forall q \in [2, \infty),
\end{equation}
which is the primary focus of the remainder of this subsection.

To this end, we first note that by the regularity property of the mild solution \( y \) as stated in Proposition \ref{prop:y-regu}, the stability property of \( P_h \) as stated in Lemma \ref{lem:Ph-stab}, the standard approximation properties of \( P_h \), the error estimate \eqref{eq:y-yh}, the growth estimate of \( F \) in \eqref{eq:F-growth-H1}, and the inverse estimate, we deduce the following stability assertion for \( y_h \):
\begin{equation}
  \label{eq:yh-regu}
  \begin{aligned}
    & \norm{y_h}_{L^p(\Omega; C([0,T]; \dot H_h^{2-\epsilon, q}))}
    + \norm{P_h y_h^3}_{L^p(\Omega; C([0,T]; \dot H_h^{1, q}))} 
    + \norm{P_h F(y_h)}_{L^p(\Omega; C([0,T]; \dot H_h^{1, q}))} \\
    & \quad \text{is uniformly bounded w.r.t.~\( h \) for all \( p, q \in (2, \infty) \) and \( \epsilon > 0 \)}.
  \end{aligned}
\end{equation}
With this stability assertion and the inequalities \eqref{eq:etDeltah} and \eqref{eq:I-etDeltah}, a routine calculation yields the following estimates:
\begin{align}
& \mathbb{E}\biggl[ \sum_{j=0}^{J-1} \int_{t_j}^{t_{j+1}}
\norm{y_h(t) - y_h(t_j)}_{L^q}^p \, \mathrm{d}t\biggr]
\leqslant c\tau^{p/2},
\quad \forall p, q \in (2, \infty),
\label{eq:yh-err1} \\
& \mathbb{E}\biggl[ \sum_{j=0}^{J-1} \int_{t_j}^{t_{j+1}}
\norm{y_h(t) - y_h(t_{j+1})}_{L^q}^p \, \mathrm{d}t\biggr]
\leqslant c\tau^{p/2},
\quad \forall p, q \in (2, \infty),
\label{eq:yh-err2} \\
& \mathbb{E}\biggl[\sum_{j=0}^{J-1} \int_{t_j}^{t_{j+1}}
\norm{y_h(t) - y_h(t_{j+1})}_{\dot H_h^{1/2 + 1/p, q}}^p \, \mathrm{d}t\biggr]
\leqslant c\tau^{p/2},
\quad \forall p, q \in (2, \infty),
\label{eq:yh-err3} \\
& \mathbb{E}\biggl[\sum_{j=0}^{J-1} \int_{t_j}^{t_{j+1}}
\norm{y_h^3(t) - y_h^3(t_{j+1})}_{L^q}^p \, \mathrm{d}t\biggr]
\leqslant c\tau^{p/2},
\quad \forall p, q \in (2, \infty).
\label{eq:yh-err4}
\end{align}

  % \item In Subsection \ref{sssec:1} we derive the following estimates:
% \begin{align}
% & \sum_{j=0}^{J-1} \int_{t_j}^{t_{j+1}}
% \mathbb E\norm{y_h(t) - y_h(t_j)}_{L^q}^p \, \mathrm{d}t
% \leqslant c\tau^{p/2},
% \quad \forall p,q \in (2,\infty),
% \label{eq:yh-err1} \\
% & \sum_{j=0}^{J-1} \int_{t_j}^{t_{j+1}}
% \mathbb E\norm{y_h(t) - y_h(t_{j+1})}_{L^q}^p \, \mathrm{d}t
% \leqslant c\tau^{p/2},
% \quad \forall p,q \in (2,\infty),
% \label{eq:yh-err2}\\
% & \sum_{j=0}^{J-1} \int_{t_j}^{t_{j+1}}
% \mathbb E\norm{y_h(t) - y_h(t_{j+1})}_{\dot H_h^{1/2+1/p,q}}^p \, \mathrm{d}t
% \leqslant c\tau^{p/2},
% \quad \forall p,q \in (2,\infty),
% \label{eq:yh-err3} \\
% & \sum_{j=0}^{J-1} \int_{t_j}^{t_{j+1}}
% \mathbb E\norm{y_h^3(t) - y_h^3(t_{j+1})}_{L^q}^p
% \leqslant c\tau^{p/2},
% \quad \forall p,q \in (2,\infty).
% \label{eq:yh-err4}
% \end{align}

% Nevertheless, it is important to acknowledge that the norm of $ y_h $ in
% the space $L^p(\Omega;C([0,T];\dot H_h^{1,q}))$ depends on the
% regularity parameters of the spatial triangulation $ \mathcal K_h $.

The structure of the remaining proof is outlined as follows:
\begin{enumerate}{}
  \item In Subsection \ref{sssec:2}, we derive the error estimate
    \begin{equation}
      \mathbb E\biggl[\tau \sum_{j=0}^J \norm{y_h(t_j) - Y_j}_{L^2}^p\biggr]
      \leqslant c\tau^{p/2}, \quad \forall p \in [2,\infty).
      \label{eq:convL2} 
    \end{equation}
  \item In Subsection \ref{sssec:love}, we establish the following stability property:
    \begin{equation}
      \label{eq:Yj-lplq-stab}
      \mathbb E\bigg[\tau \sum_{j=1}^J \norm{Y_j}_{L^q}^p\biggr]
      \text{ is uniformly bounded w.r.t.~$h$ and $\tau$ for all $ p,q \in (2,\infty)$}.
    \end{equation}
  \item In Subsection \ref{sssec:hate}, we establish the error estimate
    \begin{equation}
      \label{eq:love-3}
      \mathbb E\biggl[\tau \sum_{j=0}^J \norm{Y_j - y_h(t_j)}_{L^q}^p\biggr]
      \leqslant c\tau^{p/2}, \quad 
      \forall p,q \in (2,\infty).
    \end{equation}
  \item Finally, in Subsection \ref{sssec:3}, by utilizing
    \cref{eq:yh-regu,eq:yh-err1,eq:yh-err2,eq:yh-err3,eq:yh-err4,eq:convL2,eq:Yj-lplq-stab,eq:love-3},
    along with the discrete deterministic maximal $L^p$-regularity
    estimate from \cref{lem:DMLP}, the discrete stochastic maximal $L^p$-regularity
    estimate from \cref{pro:DSMLP}, and the stability estimate
    from \cref{prop:DMINEQ}, we conclude the proof for the desired
    error estimate \eqref{eq:convLpLq}.
\end{enumerate}

\subsubsection{Proof of the error estimate \texorpdfstring{\cref{eq:convL2}}{}}
\label{sssec:2}
For each $ 0 \leqslant j \leqslant J $, let \( e_j := Y_j - y_h(t_j) \). It is clear that
\begin{equation}
  \begin{aligned}
    \label{eq:ej}
    e_{j+1} - e_j
    &= \int_{t_j}^{t_{j+1}}
    \Delta_h\big[Y_{j+1} - y_h(t)\big] + Y_j - y_h(t) + P_hy_h^3(t) - P_hY_{j+1}^3
    \, \mathrm{d}t \\
    & \quad {} + P_h \int_{t_j}^{t_{j+1}} F(Y_j) - F(y_h(t))
    \, \mathrm{d}W_H(t), \quad \forall 0 \leqslant j < J.
  \end{aligned}
\end{equation}
We introduce the decomposition \( e_j = \xi_j + \eta_j \) for each \( 0 \leqslant j \leqslant J \),
defining the sequence \( (\xi_j)_{j=0}^J \) as follows:
\[
  \begin{cases}
    \xi_{j+1} - \xi_j = \tau \Delta_h \xi_{j+1} + 
    \int_{t_j}^{t_{j+1}} \Delta_h \big[y_h(t_{j+1}) - y_h(t)\big] \, \mathrm{d}t,
    \quad 0 \leqslant j < J, \\
    \xi_0 = 0.
  \end{cases}
\]
It can be readily verified that
\begin{equation}
  \begin{aligned}
    \label{eq:etaj}
    \eta_{j+1} - \eta_j
    &= \int_{t_j}^{t_{j+1}} \left[
      \Delta_h\eta_{j+1} +
      Y_j - y_h(t) + P_hy_h^3(t) - P_hY_{j+1}^3
    \right] \, \mathrm{d}t \\
    & \quad {} + P_h \int_{t_j}^{t_{j+1}} F(Y_j) - F(y_h(t))
    \, \mathrm{d}W_H(t), \quad \forall 0 \leqslant j < J.
  \end{aligned}
\end{equation}
Moreover, by \cref{lem:DMLP} and \cref{eq:yh-err2} we obtain
\begin{equation}
  \label{eq:xi-lplq}
  \mathbb E\biggl[\tau\sum_{j=0}^J \norm{\xi_j}_{L^q}^p\biggr]
  \leqslant c\tau^{p/2}, \quad \forall p,q \in (2,\infty).
\end{equation}
The remainder of the proof is organized into the following four steps.

\textbf{Step 1.} Let us prove that, for any $ m \in \mathbb N_{>0} $ and $ 0 \leqslant j < J $,
  \begin{equation}
    \label{eq:goal1}
    \begin{aligned}
      & \mathbb E\norm{\eta_{j+1}}_{L^2}^{2m}
      \leqslant 
      (1 + c\tau) \mathbb E \normB{
        \eta_j+P_h\int_{t_j}^{t_{j+1}} F(Y_j) - F(y_h(t)) \, \mathrm{d}W_H(t)
      }_{L^2}^{2m}  \\
      & \quad {} + c\tau\mathbb E\norm{\eta_j}_{L^2}^{2m} + c\tau\mathbb E\norm{\xi_j}_{L^2}^{2m}
      + c\mathbb E \big[ A_{j,m} + B_{j,m} \big],
    \end{aligned}
  \end{equation}
where
\begin{align*}
  A_{j,m} &:= \tau\dual{\xi_{j+1}^2, \, y_h^2(t_{j+1}) + Y_{j+1}^2}^{m}, \\
  B_{j,m} &:= \int_{t_j}^{t_{j+1}} \norm{y_h(t) - y_h(t_{j})}_{L^2}^{2m}
  + \norm{y_h^3(t) - y_h^3(t_{j+1})}_{L^{2}}^{2m} \, \mathrm{d}t.
\end{align*}
To this end, fix any $ m \in \mathbb N_{>0} $ and $ 0 \leqslant j < J $.
Consider the function
\begin{align*}
  g(\theta)
      &:= \eta_j + \theta \int_{t_j}^{t_{j+1}}
      \Big[
        \Delta_h\eta_{j+1} + Y_j - y_h(t) +
        P_h y_h^3(t) - P_hY_{j+1}^3
      \Big] \, \mathrm{d}t \\
      & \qquad {} + P_h\int_{t_j}^{t_{j+1}} F(Y_j) - F(y_h(t)) \, \mathrm{d}W_H(t),
      \quad \theta \in [0,1].
\end{align*}
For any \( \theta \in [0,1] \), differentiating
\( \norm{g(\theta)}_{L^2}^{2m} \) with respect to \( \theta \) yields
\begin{align*}
\frac{\mathrm{d}}{\mathrm{d}\theta} \norm{g(\theta)}_{L^2}^{2m}
& = 2m \norm{g(\theta)}_{L^2}^{2m-2} \dual{g(\theta), g'(\theta)} \\
& = 2m \norm{g(\theta)}_{L^2}^{2m-2} \dual{g(1) - (1-\theta) g'(\theta), g'(\theta)} \\
& \leqslant 2m \norm{g(\theta)}_{L^2}^{2m-2}
\dual{g(1) , g'(\theta)}.
\end{align*}
Since $ g(1) = \eta_{j+1} $, we have for any $ \theta \in [0,1] $ that
\begin{align*}
  \frac{\mathrm{d}}{\mathrm{d}\theta}
  \norm{g(\theta)}_{L^2}^{2m} \leqslant
  2m \norm{g(\theta)}_{L^2}^{2m-2} \biggl[ 
    \int_{t_j}^{t_{j+1}} \dual{\eta_{j+1}, \, y_h^3(t_{j+1}) - Y_{j+1}^3}
    \, \mathrm{d}t \\
    + \int_{t_j}^{t_{j+1}} \dualb{
      \eta_{j+1}, \, \Delta_h\eta_{j+1} + Y_j  - y_h(t)
      + y_h^3(t) - y_h^3(t_{j+1})
  } \, \mathrm{d}t \biggr].
\end{align*}
Using the fact $ Y_i - y_h(t_i) = \xi_i +\eta_i $ for all $ 0 \leqslant i \leqslant J $,
an elementary calculation yields
\begin{align*}
  & \frac{\mathrm{d}}{\mathrm{d}\theta}
  \norm{g(\theta)}_{L^2}^{2m} \leqslant
  2m \norm{g(\theta)}_{L^2}^{2m-2} \biggl[ 
    \int_{t_j}^{t_{j+1}} \dual{\xi_{j+1}^2, \, y_h^2(t_{j+1}) + Y_{j+1}^2} \, \mathrm{d}t \\
  & \qquad\quad {} + \int_{t_j}^{t_{j+1}}
    \norm{\eta_j}_{L^2}^2 + \norm{\xi_j}_{L^2}^2
    + \norm{y_h(t) - y_h(t_j)}_{L^2}^2
    + \norm{y_h^3(t) - y_h^3(t_{j+1})}_{L^2}^2
  \, \mathrm{d}t \biggr].
\end{align*}
Hence, applying Young's inequality and using the definitions of
\(A_{j,m}\), and \( B_{j,m} \), we obtain the following differential inequality: for any $ \theta \in [0,1] $,
  \begin{align*}
    \frac{\mathrm{d}}{\mathrm{d}\theta}
    \norm{g(\theta)}_{L^2}^{2m}
    &\leqslant
    c\tau \norm{g(\theta)}_{L^2}^{2m} +
    c\tau \norm{\eta_j}_{L^2}^{2m}
    + c\tau \norm{\xi_j}_{L^2}^{2m}
    + c\big(A_{j,m} + B_{j,m}\big).
  \end{align*}
Employing Gronwall's lemma subsequently gives
\[
\norm{g(1)}_{L^2}^{2m} \leqslant (1+c\tau)\norm{g(0)}_{L^2}^{2m}
+ c\tau \norm{\eta_j}_{L^2}^{2m}
+ c\tau \norm{\xi}_{L^2}^{2m}
+ c\left(A_{j,m} + B_{j,m}\right).
\]
Taking the expectations on both sides yields
\[
\mathbb{E}\norm{g(1)}_{L^2}^{2m} \leqslant (1+c\tau)\mathbb{E}\norm{g(0)}_{L^2}^{2m}
+ c\tau\mathbb{E}\norm{\eta_j}_{L^2}^{2m}
+ c\tau\mathbb{E}\norm{\xi_j}_{L^2}^{2m}
+ c\mathbb{E}\big[A_{j,m} + B_{j,m}\big].
\]
The desired claim \cref{eq:goal1} then follows by recognizing that 
$ g(1) = \eta_{j+1} $ and $ g(0) = \eta_j + P_h\int_{t_j}^{t_{j+1}}
F(Y_j) - F(y_h(t)) \, \mathrm{d}W_H(t) $.

% Upon integrating this differential inequality over \([0,s]\) for any \(s \in [0,1]\), we attain
% \[
%   \norm{g(s)}_{L^2}^{2m} \leqslant \norm{g(0)}_{L^2}^{2m} +
%   c\int_0^s \big[
%     \tau \norm{g(\theta)}_{L^2}^{2m} +
%     \tau \norm{\eta_j}_{L^2}^{2m} + \tau\norm{\xi_j}_{L^2}^{2m} + A_{j,m} + B_{j,m}
%   \big] \, \mathrm{d}\theta.
% \]

\textbf{Step 2.}
Fix any $ m \in \mathbb N_{>0} $ and $ 0 \leqslant j < J $.
Using H\"older's inequality, Young's inequality, and the fact
\[
  \mathbb E \biggl[
    \norm{\eta_j}_{L^2}^{2m-2} \dualB{\eta_j, \, \int_{t_j}^{t_{j+1}} F(Y_j) - F(y_h(t)) \, \mathrm{d}W_H(t)}
  \biggr] = 0,
\]
a straightforward calculation yields
\begin{align*}
  & \mathbb E\normB{
    \eta_j +  P_h \int_{t_j}^{t_{j+1}} F(Y_j) - F(y_h(t)) \, \mathrm{d}W_H(t)
  }_{L^2}^{2m} \\
  \leqslant{}
  & (1+c\tau)\mathbb E\norm{\eta_j}_{L^2}^{2m} +
  c\tau^{1-m}\mathbb E \normB{
    P_h\int_{t_j}^{t_{j+1}} F(Y_j) - F(y_h(t)) \, \mathrm{d}W_H(t)
  }_{L^2}^{2m}.
\end{align*}
For the second term on the right-hand side of this inequality,
we can use the fact that $ P_h $ is the $ L^2 $-orthogonal projection operator onto $X_h$.
Additionally, we apply \cref{lem:stoch}, H\"older's inequality, and the Lipschitz
continuity of $ F $ from \cref{eq:F-lips} to deduce that
\begin{align*}
  \tau^{1-m}\mathbb E \normB{
    P_h\int_{t_j}^{t_{j+1}} F(Y_j) - F(y_h(t)) \, \mathrm{d}W_H(t)
  }_{L^2}^{2m} \leqslant
  c\int_{t_j}^{t_{j+1}}
  \mathbb E \norm{Y_j - y_h(t)}_{L^2}^{2m}
  \, \mathrm{d}t.
\end{align*}
Consequently,
\begin{align*}
  & \mathbb E\normB{
    \eta_j +  P_h \int_{t_j}^{t_{j+1}} F(Y_j) - F(y_h(t)) \, \mathrm{d}W_H(t)
  }_{L^2}^{2m} \\
  \leqslant{}
  & (1+c\tau)\mathbb E\norm{\eta_j}_{L^2}^{2m} +
  c \int_{t_j}^{t_{j+1}} \mathbb E \norm{Y_j-y_h(t)}_{L^2}^{2m} \, \mathrm{d}t.
\end{align*}
Using the fact \(Y_j - y_h(t_j) = \xi_j + \eta_j \), along with
the definition of \(B_{j,m}\), we readily conclude that,
for any $ m \in \mathbb N_{>0} $ and $ 0 \leqslant j < J $,
  \begin{equation}
    \label{eq:goal2}
    \begin{aligned}
      & \mathbb E \normB{
        \eta_j + P_h \int_{t_j}^{t_{j+1}}
        F(Y_j) - F(y_h(t)) \, \mathrm{d}W_H(t)
      }_{L^2}^{2m} \\
      \leqslant{} &
      (1+c\tau) \mathbb E \norm{\eta_j}_{L^2}^{2m} +
      c\tau \mathbb E \norm{\xi_j}_{L^2}^{2m} +
      c\mathbb E B_{j,m}.
    \end{aligned}
  \end{equation}

\textbf{Step 3.}
By combining \cref{eq:goal1} and \cref{eq:goal2}, 
we arrive at the conclusion that for any \( m \in \mathbb{N}_{>0} \) and for all \( 0 \leqslant j < J \),
\[
\|\eta_{j+1}\|_{L^{2m}(\Omega;L^2)}^{2m} \leqslant
(1 + c\tau) \|\eta_j\|_{L^{2m}(\Omega;L^2)}^{2m}
+ c\tau \mathbb{E}\norm{\xi_j}_{L^2}^{2m}
+ c\mathbb{E}\big[A_{j,m} + B_{j,m}\big].
\]
Applying the discrete Gronwall's lemma and recalling that \(\eta_0 = 0\), we deduce that
\[
\max_{0 \leqslant j \leqslant J} \|\eta_j\|_{L^{2m}(\Omega;L^2)}^{2m}
\leqslant c \sum_{j=0}^{J-1} \mathbb{E}(\tau\norm{\xi_j}_{L^2}^{2m} + A_{j,m} + B_{j,m}), \quad \forall m \in \mathbb{N}_{>0}.
\]
Taking into account the inequalities \cref{eq:yh-err1,eq:yh-err4,eq:xi-lplq}, we then obtain that
\begin{equation}
\label{eq:etaj-l2}
\max_{0 \leqslant j \leqslant J} \|\eta_j\|_{L^{2m}(\Omega;L^2)}^{2m}
\leqslant c\tau^m + c\mathbb E\Big[ \sum_{j=0}^{J-1} A_{j,m} \Big], \quad \forall m \in \mathbb{N}_{>0}.
\end{equation}

\textbf{Step 4.}
Employing an analogous argument to that in Steps 1 and 2,
and using the stability property of \( P_h \) as stated in Lemma \ref{lem:Ph-stab},
as well as the condition \( v \in W_0^{1,\infty}(\mathcal O) \cap W^{2,\infty}(\mathcal O) \),
we can conclude that (see, also, \cite[Lemma~4.1]{Prohl2018}) for any \( m \in \mathbb N_{>0} \),
\[
\mathbb E \Bigl[ \tau\sum_{j=0}^J \norm{Y_j}_{L^2}^{2m} \Bigr]
+ \mathbb E \Bigl[ \tau \sum_{j=0}^J \norm{Y_j}_{L^4}^{4} \Bigr]
\]
is uniformly bounded with respect to \( h \) and \( \tau \). By Hölder's inequality, we further infer that for any \( p_0 \in [2,\infty) \), there exists \( q_0 \in (2,\infty) \) such that
\[
\mathbb E \Bigl[
  \tau \sum_{j=1}^J \norm{Y_j}_{L^{q_0}}^{p_0}
\Bigr]
\]
is uniformly bounded with respect to \( h \) and \( \tau \).
In view of this assertion and \cref{eq:xi-lplq}, we can use Hölder's inequality to infer that
\[
\mathbb E\Big[
  \tau\sum_{j=0}^{J-1} \dual{\xi_{j+1}^2, \, Y_{j+1}^2}^m
\Big] \leqslant c\tau^{m}, \quad \forall m \in \mathbb N_{>0}.
\]
Similarly, using the stability estimate of \( y_h \) in \cref{eq:yh-regu}, we have
\[
\mathbb E\Big[
  \tau\sum_{j=0}^{J-1} \dual{\xi_{j+1}^2, \, y_h^2(t_{j+1})}^m
\Big] \leqslant c\tau^{m}, \quad \forall m \in \mathbb N_{>0}.
\]
Consequently,
\[
\mathbb E\Big[
  \sum_{j=0}^{J-1} A_{j,m}
\Big] \leqslant c\tau^{m}, \quad \forall m \in \mathbb N_{>0},
\]
which, combined with \cref{eq:etaj-l2}, further implies that
\[
\max_{0 \leqslant j \leqslant J} \norm{\eta_j}_{L^{2m}(\Omega;L^2)}^{2m}
\leqslant c\tau^m, \quad \forall m \in \mathbb N_{>0}.
\]
Combining this inequality with \cref{eq:xi-lplq}, and noting that \( Y_j - y_h(t_j) = \xi_j + \eta_j \)
for all \( 0 \leqslant j \leqslant J \), we confirm the validity of \cref{eq:convL2} for all
\( p = 2m \), where \( m \) is a positive integer. The generalization to \( p \in [2,\infty) \)
is achieved by invoking Hölder's inequality. This completes the proof of the error estimate \cref{eq:convL2}.

% whereupon invoking the discrete Gronwall's lemma, along with the initialization condition \( e_0 = 0 \), leads to the estimation
% \[
% \max_{0 \leqslant j \leqslant J} \norm{e_j}_{L^{2m}(\Omega;L^2)}^{2m}
% \leqslant c \sum_{j=0}^{J-1} \mathbb{E}\big(A_{j,m} + B_{j,m} + C_{j,m}\big), \quad \text{for all } m \in \mathbb{N}_{>0}.
% \]
% In light of Equations (\ref{eq:yh-err1}), (\ref{eq:yh-err3}), and (\ref{eq:yh-err4}),
%this confirms the validity of the convergence assertion (\ref{eq:convL2}) for
% all \( p = 2m \) where \( m \) is a positive integer. The extension to the general
% case \( p \in [2,\infty) \) is then accomplished through an application of Hölder's inequality, thereby accomplishing the proof of the error estimate (\ref{eq:convL2}).

\subsubsection{Proof of the stability property \texorpdfstring{\cref{eq:Yj-lplq-stab}}{}}
\label{sssec:love}
From \cref{eq:Y}, we decompose the numerical solution into three components:
\[
  Y_j = \eta_j^{(0)} + \eta_j^{(1)} + \eta_j^{(2)}
  \quad\text{for all } 0 \leqslant j \leqslant J.
\]
These components are defined by the recursive relations:
\[
  \left\{
    \begin{array}{ll}
      \eta^{(0)}_{j+1} - \eta^{(0)}_j = \tau \Delta_h \eta^{(0)}_{j+1},
      & \quad 0 \leqslant j < J, \\
      \eta_0^{(0)} = P_hv, 
    \end{array}
    \right.
\]
\[
  \left\{
    \begin{array}{ll}
      \eta^{(1)}_{j+1} - \eta^{(1)}_j = \tau \Delta_h \left( \eta^{(1)}_{j+1} + Y_j - P_hY_{j+1}^3 \right),
      & \quad 0 \leqslant j < J, \\
      \eta_0^{(1)} = 0, 
    \end{array}
    \right.
\]
and
\[
  \left\{
    \begin{array}{ll}
      \eta^{(2)}_{j+1} - \eta^{(2)}_j = \tau \Delta_h \eta^{(2)}_{j+1} + P_h \int_{t_j}^{t_{j+1}} F(Y_j) \, \mathrm{d}W_H(t),
      & \quad 0 \leqslant j < J, \\
      \eta_0^{(2)} = 0.
    \end{array}
    \right.
\]

Given the identity $\eta_j^{(0)} = (I -\tau\Delta_h)^{-j}P_hv$ for $0 \leqslant j \leqslant J$,
and using the stability property of $ P_h $ from \cref{lem:Ph-stab}
along with the resolvent estimate \cref{eq:resolvent} with $\alpha = \beta = 0$, we assert that
\begin{equation} 
    \label{eq:eta0-bound}
    \mathbb E\biggl[\tau\sum_{j=1}^J \norm{\eta_j^{(0)}}_{L^q}^p\biggr]
    \text{ is uniformly bounded w.r.t.~$h$ and $\tau$ for all $ p,q \in (2,\infty)$}.
\end{equation}
For any \( p \in (2,\infty) \), the stability property of $ P_h $ from
\cref{lem:Ph-stab} and the growth property of $ F $ in \cref{eq:F-growth} imply that
\begin{align*}
  \mathbb E\biggl[\tau\sum_{j=0}^{J} \left(
     \norm{Y_j-P_hY_j^3}_{L^2}^p +
     \norm{P_hF(Y_j)}_{\gamma(H,L^6)}^p
   \right)
 \biggr]
   &\leqslant
   c + c \mathbb E\biggl[
     \tau\sum_{j=0}^{J} \norm{Y_j}_{L^6}^{3p}
   \biggr]
\end{align*}
Moreover, it can be established that
\begin{align*}
  \mathbb E\biggl[\tau\sum_{j=0}^J \norm{Y_j}_{L^6}^{3p} \biggr]
   &\leqslant c \mathbb E\biggl[
     \tau \sum_{j=0}^J \left(
     \norm{Y_j-y_h(t_j)}_{L^6}^{3p} +
     \norm{y_h(t_j)}_{L^6}^{3p}
   \right)
 \biggr] \\
   &\leqslant c + c\norm{y_h}_{C([0,T];L^{3p}(\Omega;L^6))}^{3p},
\end{align*}
where the second inequality is justified by the error estimate \cref{eq:convL2},
the condition \( \tau \leqslant h^2 \), and the application of the inverse estimate
\(
   \norm{u_h}_{L^6} \leqslant ch^{-1} \norm{u_h}_{L^2}
\) for all \( u_h \in X_h \).
Consequently, by appealing to the stability assertion \cref{eq:yh-regu}, we conclude that
\[
  \mathbb E\biggl[\tau \sum_{j=0}^J \left(
      \norm{Y_j-P_hY_j^3}_{L^2}^p +
      \norm{P_hF(Y_j)}_{\gamma(H,L^6)}^p
    \right)
  \biggr]
\]
remains uniformly bounded with respect to \( h \) and \( \tau \)
for all \( p \in (2,\infty) \).
Using this result in conjunction with \cref{lem:DMLP,pro:DSMLP},
and taking advantage of the continuous embeddings of
\(\dot{H}_h^{2,2}\) and \(\dot{H}_h^{1,6}\) into \(L^q\) for any \(q \in [2, \infty)\),
we are able to assert that
\[
  \mathbb E\biggl[\tau \sum_{j=1}^{J}
    \left( \|\eta_j^{(1)}\|_{L^q}^p + \|\eta_j^{(2)}\|_{L^q}^p \right)
  \biggr]
\]
is also uniformly bounded with respect to \( h \) and \( \tau \) for all \( p, q \in (2, \infty) \).
Combining this finding with the assertion in \eqref{eq:eta0-bound} and
noting that
$Y_j = \eta_j^{(0)} + \eta_j^{(1)} + \eta_j^{(2)}$ for $1 \leqslant j \leqslant J$,
we confirm the stability property \eqref{eq:Yj-lplq-stab}.

\subsubsection{Proof of the error estimate \texorpdfstring{\cref{eq:love-3}}{}}
\label{sssec:hate}
Firstly, according to \cref{eq:ej}, for each
$ 0 \leqslant j \leqslant J $, we decompose the error $ Y_j - y_h(t_j) $ 
into three components:
\begin{equation}
  Y_j - y_h(t_j) = \chi_j^{(0)} + \chi_j^{(1)} + \chi_j^{(2)}.
  \label{eq:Yj-yh}
\end{equation}
The sequences $ (\chi_j^{(0)})_{j=0}^J $,
$ (\chi_j^{(1)})_{j=0}^J $, and $ (\chi_j^{(2)})_{j=0}^J $ are recursively
defined as follows:
\[
  \begin{cases}
      \chi^{(0)}_{j+1} - \chi^{(0)}_j =
      \tau \Delta_h \chi^{(0)}_{j+1} + \int_{t_j}^{t_{j+1}}
      \Delta_h\big( y_h(t_{j+1}) - y_h(t) \big)
      + Y_j - y_h(t) \, \mathrm{d}t,
      \quad 0 \leqslant j < J,  \\
      \chi_0^{(0)} = 0, 
  \end{cases}
\]
  \[
    \begin{cases}
      \chi^{(1)}_{j+1} - \chi^{(1)}_j =
      \tau \Delta_h \chi^{(1)}_{j+1} +
      P_h\int_{t_j}^{t_{j+1}} 
      y_h^3(t) - Y_{j+1}^3 \, \mathrm{d}t,
      \,\, 0 \leqslant j < J, \\
      \chi_0^{(1)} = 0, 
    \end{cases}
  \]
  and
  \[
    \begin{cases}
      \chi^{(2)}_{j+1} - \chi^{(2)}_j =
      \tau \Delta_h \chi^{(2)}_{j+1} +
      P_h\int_{t_j}^{t_{j+1}}
      F(Y_j) - F(y_h(t)) \, \mathrm{d}W_H(t),
      \quad 0 \leqslant j < J,  \\
      \chi_0^{(2)} = 0.
    \end{cases}
  \]

Secondly, by using \cref{lem:DMLP}, the continuous embedding of
$ \dot H_h^{2,2} $ into $ L^q $ for all $ q \in (2,\infty) $,
and the inequalities \cref{eq:yh-err1}, \cref{eq:yh-err2}, and \cref{eq:convL2},
we derive the following bound for $ (\chi_j^{(0)})_{j=1}^J $:
\begin{align}
  \mathbb E\biggl[\tau\sum_{j=1}^J \norm{\chi_j^{(0)}}_{L^q}^p\biggr]
    & \leqslant \mathbb E\biggl[
      \sum_{j=0}^{J-1} \int_{t_j}^{t_{j+1}}
      \norm{y_h(t) - y_h(t_{j+1})}_{L^q}^p
      + \norm{Y_j - y_h(t)}_{L^2}^p
      \, \mathrm{d}t
    \biggr] \notag \\
    & \leqslant c\tau^{p/2}, \quad
    \forall p,q \in (2,\infty).
    \label{eq:tmp0}
\end{align}

Thirdly, we proceed to estimate the sequence $(\chi_j^{(1)})_{j=1}^J$. Let us fix $p, q \in (2, \infty)$.
By the stability of $P_h$ as stated in Lemma \ref{lem:Ph-stab} and Hölder's inequality, a direct computation yields
\begin{align*}
& \mathbb E\biggl[
  \sum_{j=0}^{J-1} \int_{t_j}^{t_{j+1}} \left\| P_h y_h^3(t) - P_h Y_{j+1}^3 \right\|_{L^{3/2}}^p \, \mathrm{d}t
\Biggr] \\
& \leqslant c \left[
    \mathbb E\sum_{j=0}^{J-1} \int_{t_j}^{t_{j+1}}
    \left\| y_h(t) - Y_{j+1} \right\|_{L^2}^{2p} \, \mathrm{d}t
  \right]^{1/2} \left[
    \mathbb E\sum_{j=0}^{J-1} \int_{t_j}^{t_{j+1}}
    \left\| y_h(t) \right\|_{L^{12}}^{4p} + \left\| Y_{j+1} \right\|_{L^{12}}^{4p} 
  \, \mathrm{d}t
\right]^{1/2}.
\end{align*}
By the stability properties \eqref{eq:yh-regu} and \eqref{eq:Yj-lplq-stab}, it follows that
\begin{align*}
  \mathbb E \biggl[
    \sum_{j=0}^{J-1} \int_{t_j}^{t_{j+1}}
    \left\| P_h y_h^3(t) - P_h Y_{j+1}^3 \right\|_{L^{3/2}}^p \, \mathrm{d}t
  \biggr]
& \leqslant c \left[
    \mathbb E\sum_{j=0}^{J-1} \int_{t_j}^{t_{j+1}}
    \left\| y_h(t) - Y_{j+1} \right\|_{L^2}^{2p}
  \right]^{1/2}.
\end{align*}
Hence, by applying Lemma \ref{lem:DMLP} and the continuous embedding of $\dot{H}_h^{2, 3/2}$ into $L^q$, we deduce that
\begin{align*}
  \mathbb E\biggl[\tau \sum_{j=1}^J \norm{ \chi_j^{(1)} }_{L^q}^p\biggr]
  & \leqslant c \left[
    \mathbb E\sum_{j=0}^{J-1} \int_{t_j}^{t_{j+1}}
    \left\| y_h(t) - Y_{j+1} \right\|_{L^2}^{2p}
  \right]^{1/2}.
\end{align*}
Thus, by combining \eqref{eq:yh-err2} and \eqref{eq:convL2}, we infer the following bound for $(\chi_j^{(1)})_{j=1}^J$:
\begin{equation}
\label{eq:tmp1}
\mathbb E\biggl[\tau \sum_{j=1}^J \norm{ \chi_j^{(1)} }_{L^q}^p \biggr]
\leqslant c \tau^{p/2}, \quad \forall p, q \in (2, \infty).
\end{equation}

Fourthly, using Proposition \ref{pro:DSMLP} and the continuous embedding of $\dot{H}_h^{1,2}$ into $L^6$, we obtain that, for any $p \in (2, \infty)$,
\begin{align*}
  \mathbb E \biggl[\tau \sum_{j=1}^J \norm{ \chi_j^{(2)} }_{L^6}^p \biggr]
  & \leqslant
  c \mathbb E\biggl[
    \sum_{j=0}^{J-1} \int_{t_j}^{t_{j+1}}
    \left\| F(Y_j) - F(y_h(t)) \right\|_{\gamma(H, L^2)}^p
    \, \mathrm{d}t
  \biggr].
\end{align*}
By the Lipschitz continuity of $F$ as stated in (\ref{eq:F-lips}), along with the inequalities (\ref{eq:yh-err1}) and (\ref{eq:convL2}), we establish the following bound for $(\chi_j^{(2)})_{j=1}^J$:
\begin{align*}
  \mathbb E\biggl[\tau \sum_{j=1}^J \norm{ \chi_j^{(2)} }_{L^6}^p \biggr]
& \leqslant c \tau^{p/2}, \quad \forall p \in (2, \infty).
\end{align*}
Combining this bound with (\ref{eq:Yj-yh}), (\ref{eq:tmp0}), and (\ref{eq:tmp1}) yields the error estimate
\[
  \mathbb E\biggl[\tau \sum_{j=1}^J  \left\| Y_j - y_h(t_j) \right\|_{L^6}^p \biggr]
\leqslant c \tau^{p/2}, \quad \forall p \in (2, \infty).
\]
With this error estimate, the bound for $(\chi_j^{(2)})_{j=1}^J$ can be refined to
\begin{equation}
\label{eq:tmp2}
\mathbb E\bigg[\tau \sum_{j=1}^J \norm{ \chi_j^{(2)} }_{L^q}^p \bigg]
\leqslant c \tau^{p/2}, \quad 
\forall p, q \in (2, \infty).
\end{equation}

Finally, considering the decomposition (\ref{eq:Yj-yh}) and the identity $Y_0 = y_h(0)$,
we obtain the desired error estimate (\ref{eq:love-3}) by combining the bounds established in (\ref{eq:tmp0}), (\ref{eq:tmp1}), and (\ref{eq:tmp2}).

\subsubsection{Proof of the error estimate \texorpdfstring{\cref{eq:convLpLq}}{}}
\label{sssec:3}
 Fix any $ p \in (2,\infty) $ and $ q \in [2,\infty) $.
Let the sequences \( (\chi_j^{(0)})_{j=0}^J \),
\( (\chi_j^{(1)})_{j=0}^J \), and \( (\chi_j^{(2)})_{j=0}^J \) be
constructed as in Subsection \ref{sssec:hate}.
By their construction, for any $ 1 \leqslant j \leqslant J $,
the $ L^q $-norm of $ \chi_j^{(0)} $ can be estimated almost surely as follows:
\begin{align*}
    \norm{\chi_j^{(0)}}_{L^q}
    &\leqslant
    \sum_{k=0}^{j-1} \int_{t_k}^{t_{k+1}}
    \normB{
      (I-\tau\Delta_h)^{k-j} \big[
        \Delta_h\left( y_h(t_{k+1}) - y_h(t)\right)
        + Y_k - y_h(t)
      \big]
    }_{L^q} \, \mathrm{d}t \\
    &\leqslant
      \sum_{k=0}^{j-1} \int_{t_k}^{t_{k+1}}
      \norm{(I-\tau\Delta_h)^{k-j}}_{\mathcal{L}(\dot{H}_h^{1/2+1/p,q}, \dot{H}_h^{2,q})}
      \norm{y_h(t) - y_h(t_{k+1})}_{\dot{H}_h^{1/2+1/p,q}} \, \mathrm{d}t \\
    & \quad {} + \sum_{k=0}^{j-1} \int_{t_k}^{t_{k+1}}
      \norm{(I-\tau\Delta_h)^{k-j}}_{\mathcal{L}(\dot{H}_h^{0,q})}
      \norm{Y_k - y_h(t)}_{L^q} \, \mathrm{d}t.
\end{align*}
Applying the resolvent estimate \cref{eq:resolvent} and employing H\"older's inequality,
we obtain, for any $1 \leqslant j \leqslant J$, almost surely,
\begin{align*}
\norm{\chi_j^{(0)}}_{L^q}
    &\leqslant
    c\left[
      \sum_{k=0}^{j-1} \int_{t_k}^{t_{k+1}}
      \norm{y_h(t) - y_h(t_{k+1})}_{\dot{H}_h^{1/2+1/p,q}}^p
      + \norm{Y_k - y_h(t)}_{L^q}^p \, \mathrm{d}t
    \right]^{\frac{1}{p}}.
\end{align*}
Consequently, for the sequence $(\chi_j^{(0)})_{j=1}^J$, we derive the following bound almost surely:
\[
  \max_{1\leqslant j \leqslant J}
  \norm{\chi_j^{(0)}}_{L^q}^p
  \leqslant  c \sum_{j=0}^{J-1} \int_{t_j}^{t_{j+1}}
  \norm{y_h(t) - y_h(t_{j+1})}_{\dot{H}_h^{1/2+1/p,q}}^p
  + \norm{Y_j - y_h(t)}_{L^q}^p \, \mathrm{d}t.
\]
Similarly, for $(\chi_j^{(1)})_{j=1}^J$, we have, almost surely,
\begin{align*}
  \max_{1\leqslant j\leqslant J}
  \norm{\chi_j^{(1)}}_{L^q}^p \leqslant
    c \sum_{j=0}^{J-1} \int_{t_j}^{t_{j+1}}
      \norm{P_hy_h^3(t) - P_hY_{j+1}^3}_{L^q}^p \, \mathrm{d}t.
\end{align*}
Combining these results leads to
\begin{align*}
    &\mathbb{E} \left[ \max_{1 \leqslant j \leqslant J}
    \norm{\chi_j^{(0)} + \chi_j^{(1)}}_{L^q}^p \right] \\
  \leqslant{} &
  c\sum_{j=0}^{J-1} \mathbb{E} \left[
    \int_{t_j}^{t_{j+1}} 
    \norm{y_h(t) - y_h(t_{j+1})}_{\dot{H}_h^{1/2+1/p,q}}^p + 
    \norm{Y_j-y_h(t)}_{L^q}^p 
    + \norm{P_hy_h^3(t) - P_hY_{j+1}^3}_{L^q}^p
    \, \mathrm{d}t
  \right] \\
  \leqslant{}&
  c\tau^{p/2} + c\mathbb E\biggl[\sum_{j=0}^{J-1} \int_{t_j}^{t_{j+1}} 
  \norm{P_hy_h^3(t) - P_hY_{j+1}^3}_{L^q}^p
\, \mathrm{d}t\biggr],
\end{align*}
 by \cref{eq:yh-err3,eq:yh-err1,eq:love-3}.
Furthermore, utilizing the stability property of $ P_h $,
the error estimate \cref{eq:love-3}, the inequality \cref{eq:yh-err2},
and the stability results in \cref{eq:yh-regu} and \cref{eq:Yj-lplq-stab},
we can establish the following inequality
\[
  \mathbb E\bigg[\sum_{j=0}^{J-1} \int_{t_j}^{t_{j+1}}
  \normb{P_hy_h^3(t) - P_hY_{j+1}^3}_{L^q}^p
\, \mathrm{d}t\bigg] \leqslant c\tau^{p/2}.
\]
Therefore, we conclude that
\begin{equation}
  \label{eq:bound1}
  \normB{\max_{1 \leqslant j \leqslant J}
  \normb{\chi_j^{(0)} + \chi_j^{(1)}}_{L^q}}_{L^p(\Omega)}
  \leqslant c\tau^{1/2}.
\end{equation}

Next, for the sequence \((\chi_j^{(2)})_{j=1}^J\), by the stability estimate established
in \cref{prop:DMINEQ}, we deduce that
\begin{align*}
    \left\|\max_{1 \leqslant j \leqslant J} \|\chi_j^{(2)}\|_{L^q}\right\|_{L^p(\Omega)} 
    &\leqslant c \bigg[
      \mathbb E\sum_{j=0}^{J-1}
      \int_{t_j}^{t_{j+1}}
      \left\|P_h\big(F(y_h(t))-F(Y_j)\big)\right\|_{\gamma(H,L^q)}^p
      \, \mathrm{d}t
    \bigg]^{1/p}.
  \end{align*}
Using the stability property of $ P_h $ from \cref{lem:Ph-stab}
and the Lipschitz continuity of $ F $ given by \cref{eq:F-lips},
and applying the results from \cref{eq:yh-err1} and \cref{eq:love-3},
we further establish the following bound:
\begin{equation}
  \label{eq:bound2}
    \Big\|\max_{1 \leqslant j \leqslant J} \|\chi_j^{(2)}\|_{L^q}\Big\|_{L^p(\Omega)} 
    \leqslant c\tau^{1/2}.
\end{equation}

Finally, given the decomposition \cref{eq:Yj-yh},
the desired error estimate \cref{eq:convLpLq} follows by
combining the bounds \cref{eq:bound1} and \cref{eq:bound2}.
This completes the proof of Theorem \ref{thm:convLq}.

\section{Numerical results}\label{sec:numer}
This section is dedicated to the numerical validation of the theoretical
findings presented in this paper.
The numerical experiments are conducted
using the following model equation:
\[
  \begin{cases}
    \mathrm{d}y(t) = (\Delta y + y - y^3)(t) \, \mathrm{d}t +
    y(t) \, \mathrm{d}\beta(t), \quad 0 < t \leqslant T, \\
    y(0) = v.
  \end{cases}
\]
Here, \( \beta \) represents a given Brownian motion,
the spatial domain \( \mathcal O \) is \( (0,1)^3 \),
and the initial condition is specified by the function \( v \) defined as:
\[
  v(x) := \sin(\pi x_1) \sin(\pi x_2) \sin(\pi x_3) 
  \quad\text{for all }  x:= (x_1,x_2,x_3) \in \mathcal O.
\]

To verify the spatial accuracy as established in Theorem \ref{thm:convLq},
we set the final time to \( T = 0.01 \) and the time step to
\( \tau = 1 \times 10^{-6} \). We then use the numerical solution
computed with a spatial mesh size of \( h = \frac{1}{2^8} \) as our
reference solution. Furthermore, we perform simulations across $500$
independent paths. The numerical results presented in Figure \ref{fig:space}
agrees well with the expected spatial convergence rate of \( O(h^{2-\epsilon}) \),
as outlined in \cref{eq:convLpLq-full},
for the cases of \( p=q=2 \), \( p=q=4 \), and \( p=q=16 \).

To further affirm the temporal accuracy claimed in Theorem \ref{thm:convLq},
we set the terminal time to \( T = 0.1 \) and adjust the spatial
mesh size \( h \) to meet the condition \( h = \tau^{1/2} \),
ensuring that the temporal discretization error is predominant.
We use the numerical solution with a time step of $ \tau = T/2^{13} $
as the reference solution. As before,
we conduct simulations over $500$ independent paths. The numerical results
displayed in Figure \ref{fig:time} confirm the temporal convergence
rate of \( O(\tau^{1/2}) \), consistent with the theoretical convergence
rate specified in \cref{eq:convLpLq-full}.

% We note that our theoretical analysis requires that the spatial domain $ \mathcal O $
% is a bounded, convex domain in \(\mathbb{R}^3\) of class \(\mathcal{C}^2\).
% Here our numerical experiment uses a cubic domain for convenience.
% So far, the theory of the discrete stochastic maximal $ L^p $-regularity
% is far from being mature.

% We observe that our theoretical result requires the spatial domain \(\mathcal{O}\) to be a bounded,
% convex subset of \(\mathbb{R}^3\) with a boundary of class \(\mathcal{C}^2\).
% For the purposes of our numerical experiments, we have opted for a cubic domain due to
% its simplicity and practicality. It is worth noting that the theory of discrete stochastic
% maximal \(L^p\)-regularity remains an area of ongoing research.

In our theoretical framework, it is essential that the spatial domain \(\mathcal{O}\) is a bounded,
convex subset of \(\mathbb{R}^3\) with a boundary of class \(\mathcal{C}^2\).
For the numerical experiments conducted in this study, we have chosen a cubic domain due to
its simplicity and practicality, which facilitates both computational implementation and
the interpretation of results. It is important to note that the theory of discrete stochastic
maximal \(L^p\)-regularity remains an area of ongoing research.

\begin{figure}[ht]
  \centering
  \includegraphics[width=0.7\textwidth]{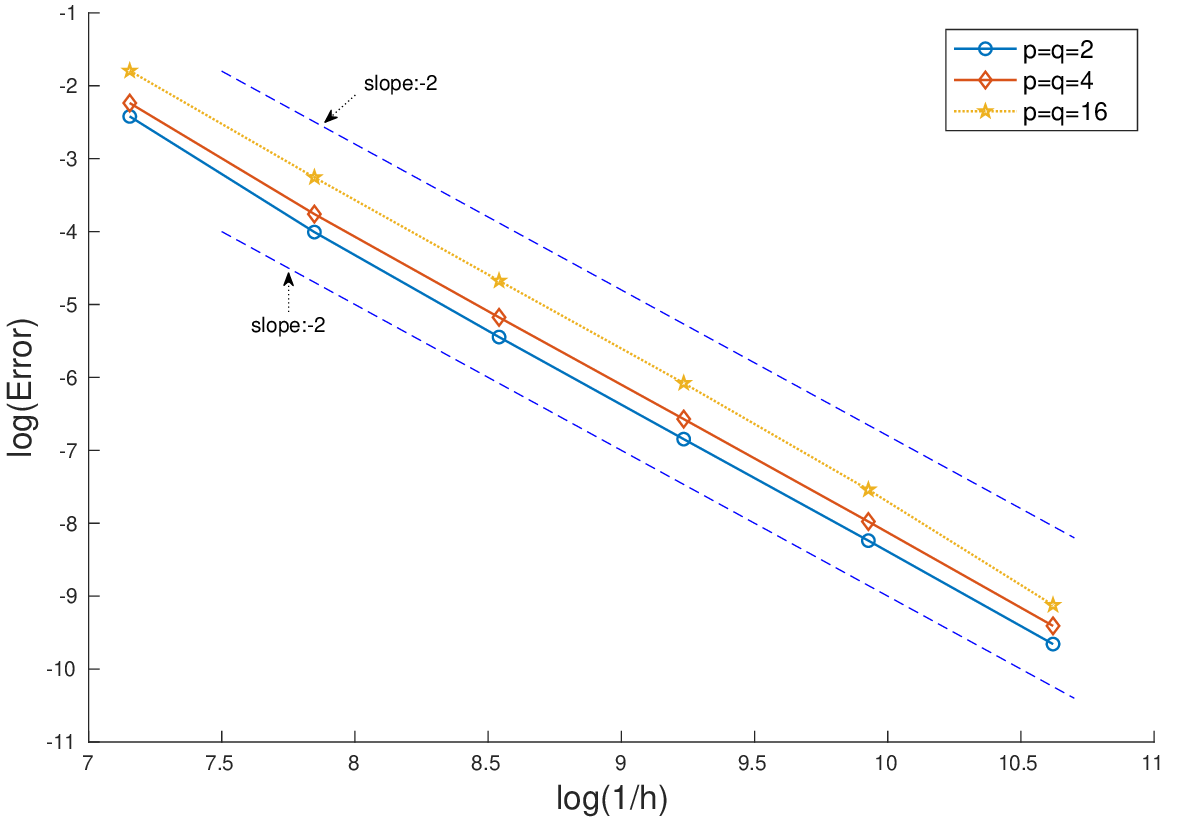}
  \caption{$\tau=1.0\times 10^{-6} $, $ T=0.01 $,}
  \label{fig:space}
\end{figure}

\begin{figure}[ht]
  \centering
  \includegraphics[width=0.7\textwidth]{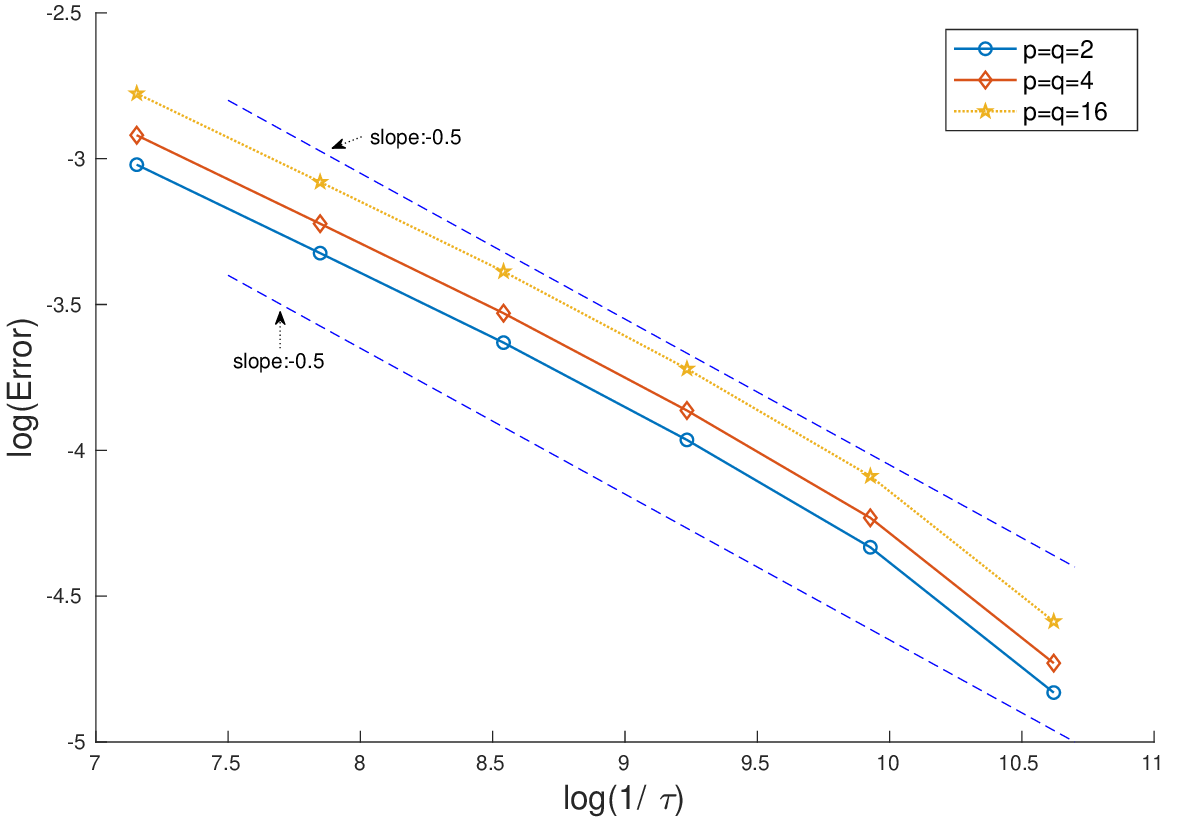}
  \caption{$T=0.1$, $ h=\tau^{1/2} $}
  \label{fig:time}
\end{figure}

%\section{Discussion of \texorpdfstring{{\boldmath$Z=X \cup Y$}}{Z = X union Y}}

% In this paper, we have derived a novel stability estimate for a 
% discrete stochastic convolution. This stability estimate is an
% independent contribution and will be instrumental in establishing
% pathwise uniform convergence estimates for fully discrete approximations
% of nonlinear stochastic parabolic equations. Using this stability
% estimate alongside the discrete stochastic maximal $L^p$ regularity
% estimate, we have achieved a pathwise uniform convergence rate of
% $ O\big( h^{2-\epsilon} + \tau^{1/2} \big) $ for a full discretization
% of a three-dimensional stochastic Allen-Cahn equation with multiplicative
% noise. Furthermore, our convergence estimate employs the general spatial
% $L^q$-norms. This theoretical result not only enhances our understanding
% of numerical methods for the stochastic Allen-Cahn equation, but also
% introduces a novel approach for the numerical analysis of nonlinear stochastic
% parabolic equations, using the discrete stochastic maixmal
% $ L^p$-regularity estimate.

\section{Conclusions}
\label{sec:conclusion}
In this paper, we have utilized discrete deterministic and stochastic maximal \(L^p\) regularity
estimates to derive a pathwise uniform convergence rate of \(O\big( h^{2-\epsilon} + \tau^{1/2} \big)\)
for a three-dimensional stochastic Allen-Cahn equation with multiplicative noise.
This convergence estimate employs general spatial \(L^q\)-norms and imposes mild
conditions on the diffusion coefficient. Our work introduces a new approach to establishing
pathwise uniform convergence within the context of general spatial \(L^q\)-norms for
finite element-based full approximations of nonlinear stochastic parabolic equations.

\bibliographystyle{plain}
% \bibliography{stochastic_optim}

\end{document}